\documentclass[10pt]{article}
\usepackage{amsmath,amssymb,amsthm}
\usepackage{color}
\usepackage[nohead,margin=1.0in]{geometry}
\usepackage{booktabs}
\usepackage{graphicx,caption,subcaption}
\usepackage{url}
\usepackage{mathtools}
\usepackage{amsmath}
\usepackage{amssymb}
\usepackage{commath}
\usepackage[normalem]{ulem}
\usepackage{tikz}
\usetikzlibrary{decorations.pathreplacing}
\usepackage{siunitx}
\usepackage{pgfplots}
\usetikzlibrary{pgfplots.groupplots}

\usepackage{nicefrac}

\usepackage{hyperref}
\hypersetup{
    colorlinks=true,
    linkcolor=blue,
    filecolor=magenta,      
    urlcolor=cyan,
    pdftitle={iResNet regularization theory},
    pdfpagemode=FullScreen,
    }

\pgfplotsset{
    compat=1.18,
    every axis plot/.append style={line width=0.8pt},
    }

\theoremstyle{plain}
\newtheorem{theorem}{Theorem}[section]
\newtheorem{remark}{Remark}[section]
\newtheorem{definition}{Definition}[section]
\newtheorem{lemma}{Lemma}[section]

\newtheorem{corollary}{Corollary}[section]

\numberwithin{equation}{section}

\newcommand{\N}{\mathbb{N}}
\newcommand{\R}{\mathbb{R}}

\newcommand{\origA}{\mathcal{A}}

\title{Invertible residual networks in the context of regularization theory for linear inverse problems}
\author{Clemens Arndt\thanks{ Center for Industrial Mathematics, University of Bremen, 28359 Bremen, Germany. Alphabetical author order. Emails: \texttt{\{carndt, adenker, heilenkoetter, iskem, pmaass, junickel\}@uni-bremen.de, \{sdittmer, tkluth\}@math.uni-bremen.de}.} \and Alexander Denker\footnotemark[1] \and S\"oren Dittmer\footnotemark[1] \footnotemark[2]\thanks{Cambridge Image Analysis Group, University of Cambridge, Cambridge CB3 0WA, UK} \and Nick Heilenk\"otter\footnotemark[1] \and Meira Iske\footnotemark[1] \and Tobias Kluth\footnotemark[1] \and Peter Maass\footnotemark[1] \and Judith Nickel\footnotemark[1]}
\date{\today}

\begin{document}

\maketitle

\begin{abstract}
Learned inverse problem solvers exhibit remarkable performance in applications like image reconstruction tasks. These data-driven reconstruction methods often follow a two-step procedure. First, one trains the often neural network-based reconstruction scheme via a dataset. Second, one applies the scheme to new measurements to obtain reconstructions. We follow these steps but parameterize the reconstruction scheme with invertible residual networks (iResNets). We demonstrate that the invertibility enables investigating the influence of the training and architecture choices on the resulting reconstruction scheme. For example, assuming local approximation properties of the network, we show that these schemes become convergent regularizations. In addition, the investigations reveal a formal link to the linear regularization theory of linear inverse problems and provide a nonlinear spectral regularization for particular architecture classes. On the numerical side, we investigate the local approximation property of selected trained architectures and present a series of experiments on the MNIST dataset that underpin and extend our theoretical findings.

%
%
%
\end{abstract}

\section{Introduction}
\label{sec:introduction}


In inverse problems, one aims to recover underlying causes from measurements by reversing the forward measurement process. Naturally, they arise in several fields, e.g., medical imaging, non-destructive testing, and PDE-based models. 
One defines inverse problems via their ill-posed nature, i.e., reversing the measurement process is discontinuous, and obtaining reasonable reconstructions requires a stable reconstruction scheme.
Usually, one defines the forward problem by possibly nonlinear forward operator $\origA : X \to Y$ mapping between Banach or Hilbert spaces. 
Linear spectral regularizations for linear problems in Hilbert space settings were already well studied over two decades ago \cite{engl1996regularization}. Also, more general nonlinear regularizations, e.g., sophisticated variational and iterative approaches, were the subject of extensive investigations for linear and nonlinear problems. We refer to the review \cite{benning2018modern} for a more general overview of this development.
More recently, the linear spectral regularization approach was generalized to diagonal frames \cite{ebner2023regularization}. Frames provide larger flexibility by allowing advantageous representations of the data, e.g., via a set of images, instead of solely singular functions provided by the operator.
Nonlinear spectral regularizations with nonlinear dependence on the measurement data were only considered in a few works, e.g., in the context of CG methods \cite[Cor. 7.4]{engl1996regularization}.

During the last decade, these research directions have been accompanied by a highly dynamic development of learning-based methods for inverse problems (see \cite{arridge2019solving} for an earlier review). 
Many promising methodological directions have arisen, such as end-to-end learned reconstructions \cite{oh2018eter,he2020radon}, learned postprocessing \cite{jin2017deep, schwab2019deep}, unrolled iteration schemes \cite{gregor2010,adler2018learned,hauptmann2018model}, plug-and-play priors \cite{venkatakrishnan2013plug,ryu2019plug,laumont2022bayesian}, learned penalties in variational approaches \cite{lunz2018adversarial,mukherjee2020learned,li2020nett,alberti2021learning,obmann2021augmented}, (deep) generative models \cite{bora2017compressed, denker2021conditional}, regularization by architecture \cite{ulyanov2020deep, dittmer2020,arndt2022regularization}, and several more directions are developed. Besides the pure method development, an increasing number of works investigate theoretical justifications using the framework of regularization theory (see the recent survey \cite{mukherjee2023learned} and the more detailed consideration in Section~\ref{sec:relatedwork}).

In the present work, we follow a general supervised learning approach for a reconstruction scheme based on the concept of invertible residual networks \cite{behrmann2019invertible}. We use this architecture to provide a nonlinear regularization scheme for which we develop the general convergence theory by exploiting a local approximation property of the underlying network. Furthermore, instead of solving the inverse problem directly, we use a training strategy that aims to approximate the forward operator. We also introduce an architecture type acting on the singular function directions, and show that it provides a data-dependent nonlinear spectral regularization which is theoretically analyzed for specific shallow architectures and it is linked to the established linear spectral regularization theory. In addition, we underpin our theoretical findings with several numerical experiments.

The manuscript is structured as follows: 
Section~\ref{sec:setting} defines the problem setting and the theoretical framework. We then investigate general regularization properties in Section~\ref{sec:local_approx_property}.
In Section~\ref{sec:specialization}, we discuss specific architectures and the outcome of a training approach aiming for approximating the forward operator. We then relate it to the classical filter-based regularization theory. Following the theoretical investigation, Section~\ref{sec:numerics} presents a series of numerical experiments. We conclude with a discussion and outlook in Section~\ref{sec:discussion}.

\subsection{Related work}
\label{sec:relatedwork}

Empirically the success of learning-based methods has been established for several applications like various imaging modalities. Theoretically, there is still a gap, e.g., regarding the convergence properties of such reconstruction schemes.
Recently an increasing number of works have examined this aspect. We recommend the excellent survey article \cite{mukherjee2023learned}, illustrating regularization properties of some learned methods \cite[Fig.3]{mukherjee2023learned}.

One early example providing a convergent learned postprocessing approach is the deep null space network \cite{schwab2019deep}. It utilizes established regularization methods and turns them into learned variants. While it replaces the concept of a minimum norm solution, it inherits convergence guarantees.
\cite{aspri2020data} investigates the convergence of regularization methods trained on a finite set of training samples without access to the entire forward operator. Specifically, the authors consider the projections on the subspaces spanned by the dataset. Assumptions on the dataset ensure desired approximation properties for the convergence analysis of the regularization method. 
The authors then also consider the data-dependent projected problem in a variational Tikhonov-type framework and derive convergence of the regularization by assuming certain approximation properties of the projection being fulfilled for the desired solution. 

Also, Tikhonov-type/variational methods and their classical regularization theory were used to obtain learned regularizations with convergence guarantees. In particular, learned convex regularizers \cite{mukherjee2020learned,mukherjee2021learning} and the network Tikhonov (NETT) approach \cite{li2020nett, obmann2021augmented} investigated this. 

Minimization of Tikhonov-type functionals is often performed by solving an equilibrium equation, e.g., obtained from first-order optimality conditions. This concept is generalized by the authors in the recent work \cite{obmann2023convergence} where a learned operator is introduced in the equilibrium equation replacing, for example, the component delivered by the regularizer. The authors formulate suitable assumptions on the learnable operator to guarantee convergence, and they also illustrate that residual networks fulfill the desired properties if the residual part in the network is contractive.

\cite{ebner2022plug} recently suggested convergence results for plug-and-play reconstructions that are conceptionally motivated by the Tikhonov theory.
They guarantee convergence by assuming the required conditions on their parameterized family of denoising operators. 

The deep image prior \cite{ulyanov2020deep} proposed exploiting a network's training and architecture as a regularization to achieve better image representation. Based on this concept of unsupervised learning from one single measurement, the authors of \cite{dittmer2020} derived an analytic deep prior framework formulated as a bilevel optimization problem. This allowed them to verify regularization properties for particular cases by exploiting relations to the classical filter theory of linear regularizations. More recently, \cite{arndt2022regularization} investigated the equivalence between Ivanov regularization and analytic deep prior. They investigated the inclusion of early stopping and proved regularization properties \cite{arndt2022regularization}.

An earlier work \cite{chung2011designing} proposed a data-driven approach to a learned linear spectral regularization; more recently, \cite{bauermeister2020learning,kabri2022convergent} considered convergence aspects. In \cite{kabri2022convergent}, the authors learn a scalar for each singular function direction in a filter-based reconstruction scheme, i.e., a linear regularization scheme. Due to the assumed properties of noise and data distributions, their training outcome is equivalent to a standard Tikhonov regularization with a data- and noise-dependent linear operator included in the penalty term, which is diagonal with respect to the system of singular functions. The authors further verify convergence results concerning the data- and noise-dependent regularization parameters.

\section{Problem setting - theoretical framework}
\label{sec:setting}

We consider linear inverse problems based on the operator equation
\begin{equation} \label{eq:inverse_problem_X_to_Y}
    \origA  x = y,
\end{equation}
where $\origA  \colon X \to Y$ is a linear and bounded operator between Hilbert spaces $X$ and $Y$. For simplification, $\|\origA \|=1$ is assumed, which can be easily obtained by a scaling of the operator. We aim to recover the unknown ground truth $x^\dagger$ as well as possible by only having access to a noisy observation $y^\delta \in Y$ such that $\|y^\delta - \origA  x^\dagger\| \leqslant  \delta$, where $\delta > 0$ is the noise level.

Instead of solving \eqref{eq:inverse_problem_X_to_Y} directly, we define $A = \origA ^* \origA $ and $z^\delta = \origA ^* y^\delta$ to get the operator equation
\begin{equation} \label{eq:inverse_problem_X_to_X}
    A x = \origA ^\ast \origA  x  =  \origA ^\ast y = z
\end{equation}
which only acts on $X$, i.e., we consider the normal equation with respect to $\origA $. 
We propose a two-step data-based approach for solving \eqref{eq:inverse_problem_X_to_X}, resp. \eqref{eq:inverse_problem_X_to_Y}, which we refer to as the \textit{iResNet reconstruction approach}:
\begin{itemize}
    \item[(I)] Supervised training of an invertible neural network $\varphi_\theta \colon X \to X$ to approximate $A$.
    \item[(II)] Using $\varphi_\theta^{-1}$
    (or $\varphi_\theta^{-1} \circ \origA ^*$, respectively) to solve the inverse problem.
\end{itemize}
In general, the supervised training covers both cases, either training on noise-free data tuples $(x^{(i)}, Ax^{(i)})$ or on noisy tuples $(x^{(i)}, Ax^{(i)}+\eta^{(i)})$ with noise $\eta^{(i)}$ for a given dataset $\{x^{(i)}\}_{i\in \{1,\hdots,N\}}\subset X$.
To guarantee the invertibility, we use a residual network of the form
\begin{equation}
    \varphi_\theta(x) = x - f_\theta(x)
\end{equation}
and restrict $f_\theta$ to be Lipschitz continuous with $\mathrm{Lip}(f_\theta) \leqslant  L < 1$. We refer to this architecture, which is illustrated in Figure~\ref{fig:resnet}, as an invertible residual network (iResNet) \cite{behrmann2019invertible}. Note that \cite{behrmann2019invertible} considers concatenations of several of these invertible blocks while we focus on a single residual block.

\begin{figure}[ht]
\centering
\includegraphics[width=0.3\textwidth]{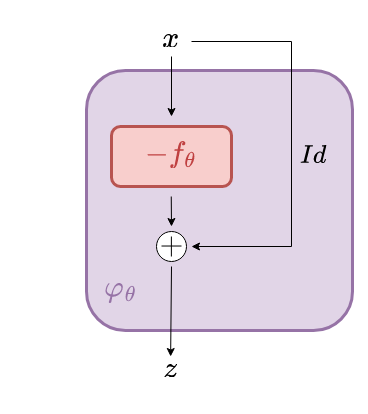}
\caption{Illustration of the residual network architecture of $\varphi_\theta$ with residual function $f_\theta$ and skip connection. }
\label{fig:resnet}
\end{figure}

The essential properties of the iResNets (like how it can be inverted) are summarized in the following lemma.
\begin{lemma}[General properties of iResNets]\label{lemma:props_iResNet}
 Let $\varphi_\theta, f_\theta: X \rightarrow X$, $\varphi_\theta(x)= x - f_\theta(x)$, where $\mathrm{Lip}(f_\theta) \leqslant  L < 1$. Then it holds:
 \begin{itemize}
     \item[(i)]
     \begin{equation}
         \mathrm{Lip}(\varphi_\theta) \leqslant L+1
     \end{equation}
     \item[(ii)] 
     \begin{equation}
         \mathrm{Lip}(\varphi_\theta^{-1}) \leqslant \frac{1}{1-L}
     \end{equation}
     \item[(iii)] For given $z\in X$, the unique $x = \varphi_\theta^{-1} (z)$ is given by 
     $x = \lim_{k\rightarrow \infty} x^{k}  $
     where 
     \begin{equation}
         x^{k+1} = f_{\theta}( x^{k }) + z
     \end{equation}
     for arbitrary and given $x^{0} \in X$.
 \end{itemize}
 
 \end{lemma}
\begin{proof}
The idea of the proof is taken from \cite[Section 2]{behrmann2019invertible}. The first assertion follows directly from 
\begin{equation}
    \mathrm{Lip}(\varphi_\theta) \leqslant \mathrm{Lip}(Id) + \mathrm{Lip}(f_\theta).
\end{equation}

Assertion (iii) follows from the fixed point theorem of Banach. Since $f_\theta$ is contractive, the iteration convergences to a unique fixed point. This fixed point $x$ fulfills 
\begin{equation}
    z = x + f_\theta(x) = \varphi(x).
\end{equation}

Finally, assertion $(ii)$ follows from
\begin{align}
    \| \varphi_\theta(x^1) - \varphi_\theta(x^2)\| &= \|x^1 - x^2 - ( f_\theta(x^1) - f_\theta(x^2)) \| \\ \notag
    &\geqslant  \|x^1 - x^2\| - \|f_\theta(x^1) - f_\theta(x^2)\| \geqslant (1-L) \|x^1 - x^2\|
\end{align}
by replacing $\varphi_\theta(x^i)$ with $z^i$ and $x^i$ with $\varphi_\theta^{-1}(z^i)$ for $i=1,2$.
\end{proof}

Besides, we assume the existence of a singular value decomposition (SVD) $(u_j, v_j, \sigma_j)_{j\in \N}$ of the operator $\origA $, i.e., $u_j \in Y$, $v_j \in X$, $\sigma_j > 0$ such that
\begin{equation}
    \origA  x = \sum_{j=1}^\infty \sigma_j \langle x, v_j \rangle u_j
\end{equation}
holds\footnote{In case of $\mathrm{dim}(\mathcal{R}(\origA )) < \infty$, the number of singular values is finite.}. All compact operators guarantee this.
With respect to $A$, the vectors $v_j$ are eigenvectors, and $\sigma_j^2$ are the corresponding eigenvalues. The system $\{v_j\}_j\in\N$ builds an orthonormal basis of $\overline{\mathcal{R}(A)} = \overline{\mathcal{R}(\origA ^\ast)} = \mathcal{N}(\origA )^\perp=\mathcal{N}(A)^\perp$. Due to the assumption $\|\origA \|=1$, it holds $\sigma_j \in (0,1]$.





\section{Regularization properties of the iResNet reconstruction}
\label{sec:local_approx_property}
In this section, we discuss regularization properties in terms of well-definedness, stability, and convergence of the general reconstruction approach defined by the family of operators $T_L=\varphi_{\theta,L}^{-1} \circ \origA ^\ast$. Note that $L \in [0, 1)$ in the constraint of the Lipschitz constant of the residual function undertakes the role of the regularization parameter by $L\to 1$, as will be made apparent in this section. 
\subsection{General architecture and local approximation property}
To highlight the essential dependence of the trained network on $L$, we write $\varphi_{\theta,L}$ and $f_{\theta,L}$ in the following. Please note that the set of possible network parameters and the underlying architecture might also depend on $L$. We thus consider a network parameter space $\Theta(L)$ depending on $L$.

\begin{lemma}[Well-definedness and stability]
For $L \in [0,1)$ and $\theta \in \Theta(L)$, let $f_{\theta,L}: X \rightarrow X$ be such that $\mathrm{Lip}(f_{\theta,L}) \leqslant  L$ and $\varphi_{\theta,L}(x)= x - f_{\theta,L}(x)$. Then, the reconstruction scheme $T_L=\varphi_{\theta,L}^{-1}\circ \origA ^\ast:Y \rightarrow X$ is well-defined and Lipschitz continuous.
\end{lemma}

\begin{proof}
Well-definedness follows immediately from Lemma \ref{lemma:props_iResNet} (iii) and Lipschitz continuity from (ii).
\end{proof}

While well-definedness and continuity (stability) immediately follow from the construction of the reconstruction method $T_L$, the convergence property requires a specific approximation property, which we must guarantee during training. We do so by introducing an index function that generalizes the concept of convergence rates (see, e.g., \cite{Schuster_2012}).

\begin{definition}
An index function is a mapping $\psi \colon \R_{\geqslant  0} \to \R_{\geqslant  0}$, which is continuous, strictly increasing and it holds $\psi(0)=0$.
\end{definition}


Now, we can formulate a convergence result for the reconstruction $T_L(y^\delta)$ for $\delta \to 0$ and a suitable a priori parameter choice $L(\delta)$.

\begin{theorem}[Convergence - local approximation property]
\label{thm:convergence}
Let $x^\dagger\in X$ be a solution of the problem $\origA x=y$ for $y\in R(\origA )$ and $y^\delta \in Y$ fulfill $\|y^\delta - y\| \leqslant \delta$. 
Furthermore, let the network parameters $\theta(L) \in \Theta(L)$ for $L \in [0,1)$ be obtained in a way that the local approximation property
\begin{equation}\label{eq:conv_trainingproperty_ass}
   \|\origA ^* \origA  x^\dagger -\varphi_{\theta(L), L}(x^\dagger) \| = \mathcal{O}((1-L)\psi(1-L)) \quad (\text{as } L \to 1)
\end{equation} 
holds for some index function $\psi$.

If $L:(0,\infty) \rightarrow [0,1)$ fulfills 
\begin{equation}\label{eq:conv_param_choice_ass}
    L(\delta) \rightarrow 1 \quad \land \quad \frac{\delta}{1-L(\delta)} \rightarrow 0 \qquad \text{for } \delta \to 0,
\end{equation}
then for $x_{L(\delta)}^\delta :=T_{L(\delta)}(y^\delta)= (\varphi_{\theta(L(\delta)),L(\delta)}^{-1} \circ \origA ^\ast)(y^\delta)$ it holds 
\begin{equation}
    \| x_{L(\delta)}^\delta - x^\dagger\| \rightarrow 0 \qquad \text{for } \delta \to 0.
\end{equation}
\end{theorem}
\begin{proof}
For improved readability we write $\varphi_\delta:=\varphi_{\theta(L(\delta)),L(\delta)}$ and $f_\delta:=f_{\theta(L(\delta)),L(\delta)}$ in the remainder of the proof. 
Using Lemma \ref{lemma:props_iResNet}, it holds
\begin{align}
     \| x_{L(\delta)}^\delta - x^\dagger\| &\leqslant  \| \varphi_\delta^{-1}(\origA ^\ast y^\delta) - \varphi_\delta^{-1}(\origA ^\ast y) \| + \| \varphi_\delta^{-1}(\origA ^\ast y) - x^\dagger \|  \notag\\
     &\leqslant \frac{\|\origA ^\ast\|}{1-L(\delta)} \|y^\delta - y\| + \| \varphi_\delta^{-1}(\origA ^\ast \origA  x^\dagger) - x^\dagger \| \notag\\
     &\leqslant \frac{\delta}{1-L(\delta)} + \| \varphi_\delta^{-1}(\origA ^\ast \origA  x^\dagger) - \varphi_\delta^{-1}(\varphi_\delta(x^\dagger)) \| \notag\\
     & \leqslant \frac{\delta}{1-L(\delta)} + \frac{1}{1-L(\delta)} \| \origA ^\ast \origA  x^\dagger -\varphi_\delta(x^\dagger) \|.
     \label{eq:convprop_auxil1}
\end{align}
The assertion follows due to the assumptions \eqref{eq:conv_trainingproperty_ass} and \eqref{eq:conv_param_choice_ass} as $\lim_{L\to 1} \psi(1-L) =0$.
\end{proof}

\begin{remark}
We can obtain a convergence rate of $\|x_{L(\delta)}^\delta - x^\dagger\| = \mathcal{O} \big( \delta^{\,\varepsilon/(1+\varepsilon)}\big)$
for $\varepsilon > 0$, if \eqref{eq:conv_trainingproperty_ass} is fulfilled with the index function $\psi(\lambda)=\lambda^\varepsilon$, i.e.,
\begin{equation}
   \|\origA ^*\origA x^\dagger -\varphi_{\theta(L),L}(x^\dagger)\| = \mathcal{O}((1-L)^{1+\epsilon}) 
\end{equation}
and by choosing $1 - L(\delta) \sim  \delta^{1/(1+\varepsilon)}.$

\end{remark}
\begin{remark}\label{rem:weaker_approx_property}
Note that the local approximation property \eqref{eq:conv_trainingproperty_ass} is a slightly stronger assumption, which is derived by exploiting the properties of the iResNet to remove the explicit dependency on the inverse within the constraint. 
A weaker but sufficient condition (which is implied by \eqref{eq:conv_trainingproperty_ass}) is 
\begin{equation}\label{eq:weaker_approx_property}
    \| \varphi_{\theta(L),L}^{-1}(\origA ^\ast \origA  x^\dagger) - x^\dagger \|  \rightarrow 0 \quad \text{as } L \to 1.
\end{equation}

\end{remark}

The local approximation property \eqref{eq:conv_trainingproperty_ass} provides a relation between the approximation capabilities of the network architecture and the regularization properties of the iResNet reconstruction scheme, which also contains a certain kind of source condition.
Note that the previous statement differs from common convergence results, which are stated not locally but for a broader range of solutions, i.e., for any $y\in \mathcal{R}(\origA)$. 
Here, the locality is strongly linked to the approximation capabilities of the underlying network, which can imply certain limitations. 
Thus, of particular interest for the convergence properties is the \textit{approximation property set}
\begin{equation}
S=\{ x\in X \,|\, \exists \text{ index function } \psi: \,x \text{ fulfills } \eqref{eq:conv_trainingproperty_ass} \}.
\end{equation}

There are no specific assumptions on the design of the training of the network $\varphi_\theta$ made by \eqref{eq:conv_trainingproperty_ass}. But since both architecture choice and training procedure are crucial for the resulting network parameters $\theta(L)$, they also have a strong impact on the local approximation property and the structure of $S$.
This is discussed in more detail in the context of the specific shallow architectures in Section~\ref{sec:specialization}. In addition, we report numerical experiments regarding the local approximation property in Section~\ref{sec:numerics}. 

A simple example illustrating the conditions under which a linear network satisfies the local approximation property is given in the following remark.
\begin{remark}[Local approximation property for linear networks]
Assume that the network $\varphi_{\theta,L}$ is linear and the prediction error is bounded by  $\|\varphi_{\theta(L(\delta)),L(\delta)}(x^{(i)}) - \origA ^* y^{\delta, (i)}\| \leqslant \zeta(\delta)$ with $\zeta(\delta)\in\R$ on a dataset $(x^{(i)}, y^{\delta, (i)})_{i=1, \dots, N}$ with $\| \origA x^{(i)} - y^{\delta, (i)} \| \leqslant \delta$. Then, for $x^\dagger$ given by $x^\dagger = \sum_{i=1}^N t_i x^{(i)}$ with $t_i\in\R$ we have
\begin{align}
   \|\origA ^* \origA  x^\dagger -\varphi_{\theta(L(\delta)),L(\delta)}(x^\dagger) \| &= \Big\| \origA ^\ast \origA  \sum_{i=1}^N t_i x^{(i)} - \sum_{i=1}^N t_i \varphi_{\theta(L(\delta)),L(\delta)} x^{(i)} \Big\| \notag\\
    &\leqslant  \sum_{i=1}^N |t_i|\, \| \origA ^\ast \origA   x^{(i)} - \varphi_{\theta(L(\delta)),L(\delta)} x^{(i)} \| \notag\\
    &\leqslant  \sum_{i=1}^N |t_i| \left( \| \origA ^\ast \origA   x^{(i)} - \origA ^*y^{\delta, (i)}\| + \|\origA ^*y^{\delta, (i)} - \varphi_{\theta(L(\delta)),L(\delta)} x^{(i)} \|\right) \notag\\
    &\leqslant  \sum_{i=1}^N |t_i|\, (\delta + \zeta(\delta)).
\end{align}
As as result, Theorem \ref{thm:convergence} implies convergence of $x_{L(\delta)}^\delta$ to $x^\dagger$ if
\begin{equation}\label{eq:param choice remark lin network}
    \frac{\delta + \zeta(\delta)}{1-L(\delta)} \to 0 \quad \text{for } \delta\to 0
\end{equation}
and if $L(\delta)$ satisfies the conditions of \eqref{eq:conv_param_choice_ass}. This example shows that in the case that the reconstruction error can be bounded by 
$\zeta(\delta)$ satisfying \eqref{eq:param choice remark lin network}, the convergence directly translates to the linear span of the training data. Consequently, for test data similar to the training data, a simple criterion, that can be numerically validated to ensure convergence to the ground truth data, is obtained. 
\end{remark}

Similar to a source condition, \eqref{eq:conv_trainingproperty_ass} also has an influence on the type of solution $x^\dagger$ one approximates in the limit $\delta \to 0$ if the operator $\origA $ has a non-trivial kernel. For example, if $\mathcal{R}(f_{\theta(L),L}) = \mathcal{N}(\origA )^\perp$, the local approximation property enforces that $P_{\mathcal{N}(\origA )}x^\dagger =0$, i.e., together with the approximation of $\origA ^\ast \origA $ one would consider the minimum norm solution.
More generally, we can formulate the following property of $S$.

\begin{lemma}
Let $x_1^\dagger, x_2^\dagger \in S$ both be solutions of $\origA x = y$ for one $y \in \mathcal{R}(\origA )$. Then it holds $x_1^\dagger =  x_2^\dagger$.
\end{lemma}

\begin{proof}
For abbreviation, we write $f_L$ and $\varphi_L$ instead of $f_{\theta(L),L}$ and $\varphi_{\theta(L),L}$.
Using $\origA x_1^\dagger = \origA  x_2^\dagger$, it holds
\begin{align}
        & \quad \, \,\|x_1^\dagger - x_2^\dagger\| \notag\\
        &\leqslant  \|x_1^\dagger - \origA ^* \origA  x_1^\dagger - f_L(x_1^\dagger)\| + \|\origA ^* \origA  x_1^\dagger + f_L(x_1^\dagger) - \origA ^* \origA  x_2^\dagger - f_L(x_2^\dagger) \| + \|\origA ^* \origA  x_2^\dagger + f_L(x_2^\dagger) - x_2^\dagger\| \notag\\
        &=  \|\varphi_L(x_1^\dagger) - \origA ^* \origA  x_1^\dagger \| + \|f_L(x_1^\dagger) - f_L(x_2^\dagger) \| + \|\origA ^* \origA  x_2^\dagger - \varphi_L(x_2^\dagger)\| \notag\\
        &\leqslant  \|\origA ^* \origA  x_1^\dagger - \varphi_L(x_1^\dagger)\| + L \|x_1^\dagger - x_2^\dagger \| + \|\origA ^* \origA  x_2^\dagger - \varphi_L(x_2^\dagger)\|.
\end{align}
Subtracting $L \|x_1^\dagger - x_2^\dagger \|$, it follows
\begin{equation}
    (1-L) \|x_1^\dagger - x_2^\dagger\| \leqslant  \|\origA ^* \origA  x_1^\dagger - \varphi_L(x_1^\dagger)\| + \|\origA ^* \origA  x_2^\dagger - \varphi_L(x_2^\dagger)\|.
    \end{equation}
Since $x_1^\dagger, x_2^\dagger$ both fulfill the local approximation property \eqref{eq:conv_trainingproperty_ass}, there exists an index function $\psi$ such that $ (1-L) \|x_1^\dagger - x_2^\dagger\| = \mathcal{O}((1-L) \psi(1-L) )$ must hold. 
This implies $\|x_1^\dagger - x_2^\dagger\| = \mathcal{O}(\psi(1-L))$, which is only possible for $x_1^\dagger = x_2^\dagger$.
\end{proof}

\subsection{Diagonal architecture}
\label{sec:diagonal_architecture}

In order to investigate relations to established and well-studied spectral regularization methods, we continue with the introduction of a particular network design, which is exploited in the remainder of the manuscript. 
The idea of spectral regularization is to decompose the data using the SVD $(u_j, v_j, \sigma_j)_{j \in \N}$ of $\origA$ and apply filter functions to the singular values. Analogously, we consider a network architecture consisting of subnetworks $f_{\theta, j}:\R \rightarrow \R$, $j\in \N$, i.e., 
\begin{equation}\label{eq:diagonal_architecture_residual}
     f_\theta(x) = \sum_{j\in \N}  f_{\theta, j}(\langle x , v_j \rangle ) v_j,
\end{equation}
where all the components $\langle x, v_j \rangle$ of the data are processed separately. We refer to this architecture as \textit{diagonal} architecture, which is also illustrated in Figure~\ref{fig:diagonal_architecture}.
The naming is motivated by the similarity of the structure of $f_\theta$ to a matrix multiplication with a diagonal matrix w.r.t the basis $\{v_j\}_{j\in\N}$. This architecture choice is, of course, less expressive than general architectures, but it has the useful benefit that this way, the iResNet can be analyzed like a filter-based regularization scheme. Besides, it enables us to simplify the infinite-dimensional inverse problem $\origA x = y$ to 1D subproblems of the form $\sigma_j \langle x, v_j \rangle = \langle y, u_j \rangle$, which will be exploited in section \ref{sec:specialization}.

In case of a diagonal architecture, the Lipschitz constraint $\mathrm{Lip}(f_\theta ) \leqslant L$ is fulfilled if $\mathrm{Lip}(f_{\theta,j} ) \leqslant L$ holds for all $j\in\N$, i.e., for any $x,y \in X$ it holds 
\begin{equation}
 \| f_\theta (x) - f_\theta(y)\|^2 = \sum_{j\in\N} | f_{\theta,j}( \langle x,v_j\rangle ) - f_{\theta,j}( \langle y,v_j\rangle ) |^2 \leqslant L^2    \sum_{j\in\N} |  \langle x-y ,v_j\rangle  |^2 \leqslant L^2 \|x-y\|^2.
\end{equation}
The local approximation property \eqref{eq:conv_trainingproperty_ass}, in this case for $x^\dagger \in \mathcal{N}(\origA)^\perp$, implies 
\begin{equation}
| f_{\theta(L),L,j}(x_j^\dagger) - (1-\sigma_j^2) x_j^\dagger|  = \mathcal{O}((1-L) \psi(1-L))  
\end{equation}
for any $j\in \N$, where $x_j^\dagger=\langle x^\dagger , v_j\rangle$.
Note that for infinite-dimensional operators, the constants in $\mathcal{O}((1-L) \psi(1-L))$ must be an $\ell^2$ sequence with respect to $j$ to obtain the implication in the opposite direction and thus equivalence.

The particular aim of using the diagonal architecture is now to find a filter function that defines a spectral regularization scheme that is equivalent to $T_L = \varphi_\theta^{-1} \circ \origA^*$. Since filter-based regularization methods are usually linear and we want to also allow for nonlinear architectures, our filter function $r_L$ must be data-dependent. Therefore, we define $r_L \colon \R_+ \times \R \to \R$ such that for $z \in \mathcal{R}(\origA ^\ast)$
\begin{align} \label{eq:data_dep_filterfunc}
    \varphi_\theta^{-1}(z) = \sum_{j \in \mathbb{N}} r_L (\sigma_j^2, \langle z, v_j \rangle) \langle z, v_j \rangle v_j, 
\end{align}
or for $y \in Y$
\begin{equation}
\label{eq:data_dep_filterfunc_y}
   T_L(y)= \varphi_\theta^{-1}(\origA ^\ast y) =  \sum_{j \in \mathbb{N}} r_L (\sigma_j^2, \sigma_j \langle y, u_j \rangle) \sigma_j \langle y, u_j \rangle v_j
\end{equation}
holds. The first argument of $r_L$ is the singular value (as usual), and the data-dependency comes via the second argument.
This can be seen as a nonlinear extension to the established filter theory of linear regularization methods \cite{rieder2013keine, louis1989inverse}, where $r_L$ would depend on $\sigma_j^2$ only.
The subnetworks thus play an important role in defining the filter functions via
\begin{equation}
\label{eq:filter_function}
(Id-f_{\theta,L,j})^{-1}(s) = r_L(\sigma_j^2,s) s.
\end{equation}
For this to be well-defined, we need to assume that all singular values $\sigma_j > 0$ have a multiplicity of one. In case of $\sigma_j = \sigma_k$ for $j \neq k$ (multiplicity greater than one) one must ensure that $f_{\theta, L, j}$ and $f_{\theta, L, k}$ also coincide. In practice, this could be realized by sharing the weights of these networks.

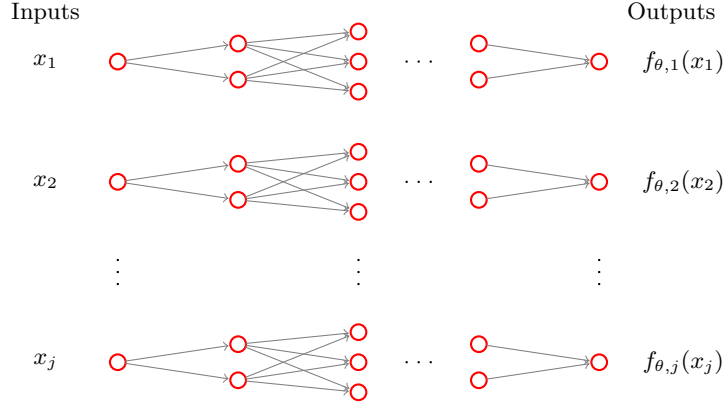
\begin{figure}[ht]
\centering
\begin{tikzpicture}[scale=0.8]
\small
\draw[color=red, thick] (0,4) circle (0.13);
\draw[->, color=gray,  shorten >=1.3mm, shorten <=1mm] (0,4) -- (2,4.3);
\draw[->, color=gray,  shorten >=1.3mm, shorten <=1mm] (0,4) -- (2,3.7);
\draw[color=red, thick] (2,4.3) circle (0.13);
\draw[color=red, thick] (2,3.7) circle (0.13);
\draw[->, color=gray,  shorten >=1.3mm, shorten <=1mm] (2,4.3) -- (4,4.5);
\draw[->, color=gray,  shorten >=1.3mm, shorten <=1mm] (2,4.3) -- (4,4);
\draw[->, color=gray,  shorten >=1.3mm, shorten <=1mm] (2,4.3) -- (4,3.5);
\draw[->, color=gray,  shorten >=1.3mm, shorten <=1mm] (2,3.7) -- (4,4.5);
\draw[->, color=gray,  shorten >=1.3mm, shorten <=1mm] (2,3.7) -- (4,4);
\draw[->, color=gray,  shorten >=1.3mm, shorten <=1mm] (2,3.7) -- (4,3.5);
\draw[color=red, thick] (4,4.5) circle (0.13);
\draw[color=red, thick] (4,4) circle (0.13);
\draw[color=red, thick] (4,3.5) circle (0.13);
\filldraw[color=black] (4.8,4) circle (0.01);
\filldraw[color=black] (5,4) circle (0.01);
\filldraw[color=black] (5.2,4) circle (0.01);
\draw[color=red, thick] (6,4.3) circle (0.13);
\draw[color=red, thick] (6,3.7) circle (0.13);
\draw[->, color=gray,  shorten >=1.3mm, shorten <=1mm] (6,4.3) -- (8,4);
\draw[->, color=gray,  shorten >=1.3mm, shorten <=1mm] (6,3.7) -- (8,4);
\draw[color=red, thick] (8,4) circle (0.13);

\draw[color=red, thick] (0,2) circle (0.13);
\draw[->, color=gray,  shorten >=1.3mm, shorten <=1mm] (0,2) -- (2,2.3);
\draw[->, color=gray,  shorten >=1.3mm, shorten <=1mm] (0,2) -- (2,1.7);
\draw[color=red, thick] (2,2.3) circle (0.13);
\draw[color=red, thick] (2,1.7) circle (0.13);
\draw[->, color=gray,  shorten >=1.3mm, shorten <=1mm] (2,2.3) -- (4,2.5);
\draw[->, color=gray,  shorten >=1.3mm, shorten <=1mm] (2,2.3) -- (4,2);
\draw[->, color=gray,  shorten >=1.3mm, shorten <=1mm] (2,2.3) -- (4,1.5);
\draw[->, color=gray,  shorten >=1.3mm, shorten <=1mm] (2,1.7) -- (4,2.5);
\draw[->, color=gray,  shorten >=1.3mm, shorten <=1mm] (2,1.7) -- (4,2);
\draw[->, color=gray,  shorten >=1.3mm, shorten <=1mm] (2,1.7) -- (4,1.5);
\draw[color=red, thick] (4,2.5) circle (0.13);
\draw[color=red, thick] (4,2) circle (0.13);
\draw[color=red, thick] (4,1.5) circle (0.13);
\filldraw[color=black] (4.8,2) circle (0.01);
\filldraw[color=black] (5,2) circle (0.01);
\filldraw[color=black] (5.2,2) circle (0.01);
\draw[color=red, thick] (6,2.3) circle (0.13);
\draw[color=red, thick] (6,1.7) circle (0.13);
\draw[->, color=gray,  shorten >=1.3mm, shorten <=1mm] (6,2.3) -- (8,2);
\draw[->, color=gray,  shorten >=1.3mm, shorten <=1mm] (6,1.7) -- (8,2);
\draw[color=red, thick] (8,2) circle (0.13);

\filldraw[color=black] (0,0.7) circle (0.01);
\filldraw[color=black] (0,0.5) circle (0.01);
\filldraw[color=black] (0,0.3) circle (0.01);

\filldraw[color=black] (4,0.7) circle (0.01);
\filldraw[color=black] (4,0.5) circle (0.01);
\filldraw[color=black] (4,0.3) circle (0.01);

\filldraw[color=black] (8,0.7) circle (0.01);
\filldraw[color=black] (8,0.5) circle (0.01);
\filldraw[color=black] (8,0.3) circle (0.01);

\draw[color=red, thick] (0,-1) circle (0.13);
\draw[->, color=gray, shorten >=1.3mm, shorten <=1mm] (0,-1) -- (2,-0.7);
\draw[->, color=gray, shorten >=1.3mm, shorten <=1mm] (0,-1) -- (2,-1.3);
\draw[color=red, thick] (2,-0.7) circle (0.13);
\draw[color=red, thick] (2,-1.3) circle (0.13);
\draw[->, color=gray, shorten >=1.3mm, shorten <=1mm] (2,-0.7) -- (4,-0.5);
\draw[->, color=gray, shorten >=1.3mm, shorten <=1mm] (2,-0.7) -- (4,-1);
\draw[->, color=gray, shorten >=1.3mm, shorten <=1mm] (2,-0.7) -- (4,-1.5);
\draw[->, color=gray, shorten >=1.3mm, shorten <=1mm] (2,-1.3) -- (4,-0.5);
\draw[->, color=gray, shorten >=1.3mm, shorten <=1mm] (2,-1.3) -- (4,-1);
\draw[->, color=gray, shorten >=1.3mm, shorten <=1mm] (2,-1.3) -- (4,-1.5);
\draw[color=red, thick] (4,-0.5) circle (0.13);
\draw[color=red, thick] (4,-1) circle (0.13);
\draw[color=red, thick] (4,-1.5) circle (0.13);
\filldraw[color=black] (4.8,-1) circle (0.01);
\filldraw[color=black] (5,-1) circle (0.01);
\filldraw[color=black] (5.2,-1) circle (0.01);
\draw[color=red, thick] (6,-0.7) circle (0.13);
\draw[color=red, thick] (6,-1.3) circle (0.13);
\draw[->, color=gray, shorten >=1.3mm, shorten <=1mm] (6,-0.7) -- (8,-1);
\draw[->, color=gray, shorten >=1.3mm, shorten <=1mm] (6,-1.3) -- (8,-1);
\draw[color=red, thick] (8,-1) circle (0.13);

\draw[] (-1.2,4.8) node {Inputs};
\draw[] (-1.2,4) node {$x_1$};
\draw[] (-1.2,2) node {$x_2$};
\draw[] (-1.2,-1) node {$x_j$};

\draw[] (9.2,4.8) node {Outputs};
\draw[] (9.4,4) node {$f_{\theta,1}(x_1)$ };
\draw[] (9.4,2) node {$f_{\theta,2}(x_2)$ };
\draw[] (9.4,-1) node {$f_{\theta,j}(x_j)$ };

\end{tikzpicture}
    \caption{\textit{Diagonal} structure of the residual function including subnetworks $f_{\theta,j}: \R \to \R$ acting on $x_j=\langle x, v_j \rangle$.}
    \label{fig:diagonal_architecture}
\end{figure}

\begin{remark}
Some authors define filter functions in terms of the original problem with respect to the generalized inverse $\origA ^\dagger$ such that
\begin{align} \label{eq:data_dep_filterfunc_original_problem}
    \varphi_\theta^{-1}(\origA ^\ast y) = \sum_{j \in \mathbb{N}} F_L (\sigma_j, \langle y, u_j \rangle) \frac{1}{\sigma_j} \langle y, u_j \rangle v_j, \qquad \text{with } F_L(\sigma,s)=\sigma^2 r_L(\sigma^2, \sigma s).
\end{align}
Note that the choice of the representation, either in terms of $r_L$ or in terms of $F_L$, depends on personal preference (compare, for example, \cite{rieder2013keine} and \cite{louis1989inverse}). In the following, we represent the filter function in terms of $r_L$.
\end{remark}

One simple nonlinear dependence on $s$ will be taken into account explicitly in the following, as it becomes relevant in later investigations. In analogy to the bias in neural network architectures, we consider a linear filter framework with bias, i.e., 
\begin{equation}\label{eq:linear_filter_with_bias}
T_L(y) = \hat{b}_L +
\sum_{j\in \N} \hat{r}_L(\sigma_j^2) \sigma_j \langle y , u_j \rangle v_j
\end{equation}
where $\hat{b}_L \in X$ is a bias term and $\hat{r}_L$ is a classical filter function. 
With additional simple assumptions on $\hat{b}_L$, this becomes a regularization method. For the sake of completeness, we provide the following result.

\begin{lemma}[Filter-based spectral regularization with bias]
\label{lem:filter_regularization_with_bias}
Let $\hat{r}_L:[0,\|\origA \|^2] \to \R$ be a piecewise continuous function and let $b_L \in X$ for $L\in [0,1)$. Furthermore, let
\begin{itemize}
    \item[(i)] $\lim_{L\to 1} \hat{r}_L(\sigma_j^2) =\frac{1}{\sigma_j^2}$ for any $\sigma_j$, $j \in \N$,
    \item[(ii)] $\exists\, 0 < C < \infty: \ \sigma_j^2\, |\hat{r}_L (\sigma_j^2) | \leqslant C$  for any $\sigma_j$, $j \in \N$ and $L\in [0,1)$, 
    \item[(iii)] $\hat{b}_L \in X$ for $L\in [0,1)$ and $\lim_{L\to 1} \| \hat{b}_L \| =0$
\end{itemize}
hold.
Let $x^\dagger\in \mathcal{N}(\origA )^\perp$ be a solution of the problem $\origA x=y$ for $y\in R(\origA )$, the operator $T_L:Y \rightarrow X$ be given by \eqref{eq:linear_filter_with_bias} and $y^\delta\in Y$ with $\|y-y^\delta\| \leqslant \delta$.
In addition, let $L:(0,\infty) \rightarrow [0,1)$ be such that 
\begin{equation}\label{eq:conv_param_choice_ass_linear_filter_with_bias}
    L(\delta) \rightarrow 1 \quad \land \quad \delta \sqrt{ \sup \big\lbrace |\hat{r}_{L(\delta)}(\sigma_i)| \, \big| \, i\in \N \big\rbrace }  \rightarrow 0
\end{equation}
as $\delta \to 0$. 
Then, $T_L$ is a regularization method and for $x_{L(\delta)}^\delta := T_{L(\delta)}(y^\delta)$ it holds $\lim_{\delta \to 0} \| x_{L(\delta)}^\delta - x^\dagger\| =0$. 
\end{lemma}
\begin{proof}
By defining $\tilde{x}_{L(\delta)}^\delta := T_{L(\delta)} (y^\delta) - \hat{b}_{L(\delta)}$ (removing the bias), we get an ordinary linear regularization method (see, e.g., \cite[Corr. 3.3.4]{rieder2013keine}). It holds
\begin{equation}
\| x_{L(\delta)}^\delta - x^\dagger\| \leqslant \| x_{L(\delta)}^\delta - \tilde{x}_{L(\delta)}^\delta\| + \| \tilde{x}_{L(\delta)}^\delta - x^\dagger\| = \| \hat{b}_{L(\delta)}\| + \| \tilde{x}_{L(\delta)}^\delta - x^\dagger\|.
\end{equation}
Since $\|\hat{b}_{L(\delta)}\| \to 0$ by assumption and $ \tilde{x}_{L(\delta)}^\delta \to x^\dagger$ by standard theory, $T_L$ is a regularization method. 
\end{proof}

\section{Approximation training of specialized architectures}
\label{sec:specialization}

Besides the architecture choice and the available data, the loss function for network training 
is an important ingredient. 
Having regularization properties in mind, a natural choice for training the iResNet is to approximate the forward operator $A \colon X \to X$, $A = \origA ^* \origA $, on a given training dataset $\{x^{(i)}\}_{i\in \{1,\hdots,N\}}  \subset X$. This is also strongly motivated by the structure of the local approximation property \eqref{eq:conv_trainingproperty_ass}  to obtain convergence guarantees.
The training of $\varphi_\theta = Id - f_\theta$ then consists of solving
\begin{equation} \label{eq:training_minproblem}
    \min_{\theta \in \Theta} l(\varphi_\theta,A)=\min_{\theta \in \Theta}  \frac{1}{N} \sum_{i} \|\varphi_\theta(x^{(i)}) - A x^{(i)} \|^2  \quad \text{s.t. } \mathrm{Lip}(f_\theta) \leqslant  L
\end{equation}
with $L < 1$ as a hyperparameter, which we refer to as \textit{approximation training}.
Here, we restrict ourselves to the noise-free case as the outcome of the approximation training is independent for $N \to\infty$ (assuming stochastic independence between data and noise). 

In the following, we analyze to which extent $\varphi_\theta^{-1}$ acts as a regularized inverse of $A$ for different simple diagonal iResNet architectures trained according to \eqref{eq:training_minproblem}.
We compute the particular network parameters $\theta$ 
and
derive the corresponding data-dependent filter function $r_L \colon \mathbb{R} \times \mathbb{R} \to \mathbb{R}$ (see Section \ref{sec:diagonal_architecture}). Additionally, we check whether the local approximation property \eqref{eq:conv_trainingproperty_ass} is fulfilled.
For this purpose, the assumptions on the architecture as well as on the dataset used for the training must be specified.

\subsection{One-parameter-network -- Tikhonov}

We begin with a very simple linear network, which only has one single scalar learnable parameter. We show that this architecture is equivalent to Tikhonov regularization if trained appropriately.

\begin{lemma}
\label{lem:one_parameter_network}
      Let $(v_j, \sigma_j^2)_j$ be the eigenvectors and eigenvalues of $A$ and let $\varphi_\theta = Id - f_\theta$ be an iResNet which solves \eqref{eq:training_minproblem} with
    \begin{itemize}
        \item[(i)] $f_\theta = k (Id - A)$, where $\theta = k$ and $\Theta = \R$ (\textbf{architecture} assumption),
        \item[(ii)] the training dataset $\{x^{(i)}\}_{i\in \{1,\hdots,N\}}$ contains at least one $x^{(i)}$ s.t.\ $A x^{(i)} \neq x^{(i)}$ (\textbf{dataset} assumption).
    \end{itemize}
    Then, the solution of \eqref{eq:training_minproblem} is $k = L$ and \eqref{eq:data_dep_filterfunc} holds with
    \begin{equation}\label{eq:filter-Tikhonov}
        r_L(\sigma^2, s) = \frac{1}{L} \cdot \frac{1}{\alpha + \sigma^2}, \qquad \alpha = \frac{1-L}{L},
    \end{equation}
    for any $s \in \R$.
\end{lemma}

The proof of the lemma can be found in Appendix~\ref{sec:proof_of_one_parameter_network}.
The derived filter function corresponds to the Tikhonov method (with reference element)
\begin{equation}
    \min_{x \in X} \| \origA x - y^\delta \|^2 + \alpha \|x - \origA^* y^\delta\|^2
\end{equation}
as one can see by considering the first-order optimality condition, which implies
\begin{equation}
    x = \sum_j \frac{1+\alpha}{\alpha + \sigma_j^2} \sigma_j \langle y^\delta, u_j \rangle v_j
\end{equation}
and using the fact that $\frac{1}{L} = 1 + \alpha$. It is illustrated in Figure~\ref{fig:Tik_filter}. We plot $\sigma^2 r_L(\sigma^2, s)$ since the multiplication of the filter function with $\sigma^2$ corresponds to the concatenation $\varphi_\theta^{-1} \circ A$ which emphasizes the regularizing effect. 

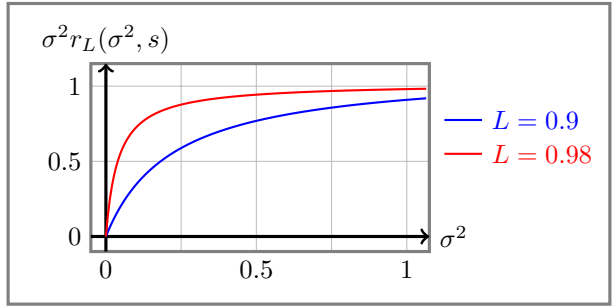
\begin{figure}[ht]
\centering
  
  
  
\begin{tikzpicture}
  \draw[step=1.0, lightgray, thin] (-0.2,-0.2) grid (4.3,2.3);
  \draw[gray, very thick] (-0.2,-0.2) rectangle (4.3,2.3);
  \draw[gray, very thick] (-1.3,-0.9) rectangle (6.7,3.1);
  \draw[->, thick, very thick] (-0.2, 0) -- (4.3, 0) node[right] {$\sigma^2$};
  \draw[->, thick, very thick] (0, -0.2) -- (0, 2.3) node[above] {$\sigma^2 r_L(\sigma^2, s)$};
  \draw[scale=2, domain=0:2.13, variable=\t, blue, thick] plot[samples=161] ({\t}, {1/0.9*0.25*\t/((1-0.9)/0.9 + 0.25*\t)});
  \draw[scale=2, domain=0:2.13, variable=\t, red, thick] plot[samples=161] ({\t}, {1/0.98*0.25*\t/((1-0.98)/0.98 + 0.25*\t)});
  
  \draw (-0.2,0) node[left]  {0};
  \draw (-0.2,1) node[left]  {0.5};
  \draw (-0.2,2) node[left]  {1};
  
  \node[below] at (0,-0.2) {$0$};
  \node[below] at (2,-0.2) {$0.5$};
  \node[below] at (4,-0.2) {$1$};
  
  \draw[blue, thick] (4.5,1.55) -- (5,1.55) node[right] {$L=0.9$};
  \draw[red, thick] (4.5,1.1) -- (5,1.1) node[right] {$L=0.98$};
\end{tikzpicture}
\caption{
The plot depicts graphs of $\sigma^2 r_L(\sigma^2, s)$ as given in \eqref{eq:filter-Tikhonov} to visualize the effect of the regularization, which is also known as standard Tikhonov regularization.}
\label{fig:Tik_filter}
\end{figure}

\begin{remark}
This particular choice of the residual function $f_\theta$ is one example where the local approximation property is not fulfilled for any $x^\dagger \in X$ except for eigenvectors corresponding to $\sigma_1^2=1$, i.e., $S=\mathrm{span} \{v_1\}$. 
But the linear nature and the specific spectral representation of the residual network allow for the verification of the weaker constraint in \eqref{eq:weaker_approx_property} being highlighted in Remark~\ref{rem:weaker_approx_property} and being sufficient for convergence. For the residual function $f_{\theta(L),L}=L(Id-A)$ and $x^\dagger \in \mathcal{N}(A)^\bot$ we obtain
\begin{align}
 \| \varphi_{\theta(L),L}^{-1}(A x^\dagger) - x^\dagger \|^2 & = \| ((Id - L(Id-A))^{-1}A  - Id) x^\dagger \|^2 \notag\\ 
 &= \sum_{j \in \N} \bigg( \underset{=\sigma_j^2 r_L(\sigma_j^2)}{\underbrace{\frac{\sigma_j^2}{1 - L (1- \sigma_j^2)}}}  - 1\bigg)^2 |\langle x^\dagger , v_j\rangle|^2,
 \end{align}
which converges to zero for $L \to 1$.

Alternatively, convergence can be verified by standard arguments following, for example, the line of reasoning in the proof of \cite[Thm. 3.3.3]{louis1989inverse}, where properties of the filter function $F_L(\sigma_j)=\sigma_j^2 r_L(\sigma_j^2)$ are exploited. Since $F_L$ defines the filter function for Tikhonov regularization and is therefore known to fulfill the desired properties, we have convergence for arbitrary $x^\dagger \in \mathcal{N}(A)^\perp$.  
\end{remark}  
\subsection{Affine linear network -- squared soft TSVD}
\label{sec:soft_TSVD}

We continue with a slightly more general affine linear architecture $f_\theta(x) = Wx + b$. It can be observed that the linear operator $W: X \to X$ only depends on $A$ and $L$ while the bias vector $b\in X$ is data dependent.

\begin{lemma}
\label{lem:squared_soft_TSVD}
    Let $(v_j, \sigma_j^2)_j$ be the eigenvectors and eigenvalues of $A$ and $\varphi_\theta = Id - f_\theta$ be an iResNet which solves \eqref{eq:training_minproblem} where
    \begin{itemize}
        \item[(i)] $f_\theta(x) = W x + b$, with $\theta = (W, b)$ and 
        \begin{equation}
            \Theta = \left\lbrace (W,b) \in L(X) \times X \, \Bigg| \, \exists (w_j)_{j\in \N},(b_j)_{j\in \N}: \ W = \sum_{j \in \mathbb{N}} w_j \langle \cdot, v_j \rangle v_j, \, b = \sum_{j \in \N} b_j v_j \right\rbrace,
        \end{equation}
        \item[(ii)] the training dataset $\{x^{(i)}\}_{i\in \{1,\hdots,N\}}$ and the mean values $\mu_j = \frac{1}{N} \sum_{i=1}^N \langle x^{(i)}, v_j \rangle$ fulfill
        \begin{equation}
            \forall j \in \N \colon \, \exists i \in \{1, ..., N\} \colon \quad \langle x^{(i)}, v_j \rangle \neq \mu_j.
        \end{equation}
    \end{itemize}
    Then, the solution of the training problem \eqref{eq:training_minproblem} is $w_j =  \min \{1-\sigma_j^2, L \}$, $b_j = \max \{0, 1-L-\sigma_j^2\} \mu_j$ and $\varphi_\theta^{-1} \circ \origA^*$ is equivalent to $T_L$ in \eqref{eq:linear_filter_with_bias} with
    \begin{equation}\label{eq:filter_affine_soft_TSVD}
        \hat{r}_L(\sigma^2) = \frac{1}{\max \{\sigma^2, 1-L \}}, \qquad \hat{b}_L = \frac{1}{1-L}\sum_{j \in \N} b_j v_j = \sum_{\sigma_j^2 < 1-L} \frac{1-L-\sigma_j^2}{1-L} \mu_j v_j.
    \end{equation}
\end{lemma}

The proof of the lemma can be found in Appendix~\ref{sec:proof_of_squared_soft_TSVD}.
The filter function is illustrated in Figure~\ref{fig:filter_soft}.
\begin{remark}
    Note that a similar filter function has been found in the context of an analytic deep prior approach for regularization \cite{dittmer2020}. The authors derived the filter function $F_L(\sigma)=\sigma^2 \hat{r}_L(\sigma^2)$, which is linear for small $\sigma$, and named it \textit{soft TSVD} (soft truncated SVD). In contrast, the filter function in \eqref{eq:filter_affine_soft_TSVD} is a polynomial of degree two for small $\sigma$, which we refer to as \textit{squared soft TSVD}.    
\end{remark}


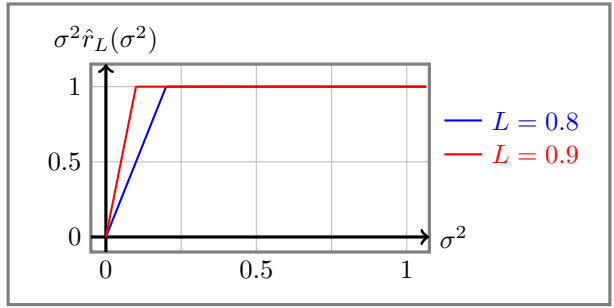
\begin{figure}[h]
\centering
  
  
  
\begin{tikzpicture}
  \draw[step=1.0, lightgray, thin] (-0.2,-0.2) grid (4.3,2.3);
  \draw[gray, very thick] (-0.2,-0.2) rectangle (4.3,2.3);
  \draw[gray, very thick] (-1.3,-0.9) rectangle (6.7,3.1);
  \draw[->, thick, very thick] (-0.2, 0) -- (4.3, 0) node[right] {$\sigma^2$};
  \draw[->, thick, very thick] (0, -0.2) -- (0, 2.3) node[above] {$\sigma^2 \hat{r}_L(\sigma^2)$};
  \draw[scale=2, domain=0:2.13, variable=\t, blue, thick] plot[samples=161] ({\t}, {0.5*\t/max(0.5*\t,1-0.8)});
  \draw[scale=2, domain=0:2.13, variable=\t, red, thick] plot[samples=161] ({\t}, {0.5*\t/max(0.5*\t,1-0.9)});
  
  \draw (-0.2,0) node[left]  {0};
  \draw (-0.2,1) node[left]  {0.5};
  \draw (-0.2,2) node[left]  {1};
  
  \node[below] at (0,-0.2) {$0$};
  \node[below] at (2,-0.2) {$0.5$};
  \node[below] at (4,-0.2) {$1$};
  
  \draw[blue, thick] (4.5,1.55) -- (5,1.55) node[right] {$L=0.8$};
  \draw[red, thick] (4.5,1.1) -- (5,1.1) node[right] {$L=0.9$};
\end{tikzpicture}
\caption{
The plot depicts graphs of $\sigma^2 \hat{r}_L(\sigma^2)$ as given in \eqref{eq:filter_affine_soft_TSVD} to visualize the effect of the regularization, which is similar to the known soft TSVD regularization.}
\label{fig:filter_soft}
\end{figure}

Due to the affine linear nature of the found reconstruction scheme, regularization properties can immediately be derived by verifying certain properties of the filter function $\hat{r}_L$ and the bias $\hat{b}_L$ using Lemma~\ref{lem:filter_regularization_with_bias}. This is summarized in the following corollary.
 
\begin{corollary}
 The filter based method \eqref{eq:linear_filter_with_bias} with $\hat{r}_L$ and $\hat{b}_L$  given by \eqref{eq:filter_affine_soft_TSVD} fulfills the properties (i)-(iii) in Lemma~\ref{lem:filter_regularization_with_bias}, i.e., in the setting of Lemma~\ref{lem:filter_regularization_with_bias} it is a convergent regularization.
\end{corollary}

The previous result provides a link to the classical theory for linear regularization operators. Besides the well-known classical theory, we further want to investigate the local approximation property \eqref{eq:conv_trainingproperty_ass}
to verify whether the convergence property can also be obtained using Theorem~\ref{thm:convergence}.
As a first step, we derive a technical source condition for $x^\dagger$, which implies the local approximation property.



\begin{lemma}\label{lem:squared_soft_local_approx}
Let the network $\varphi_\theta$ and the setting be that of
  Lemma \ref{lem:squared_soft_TSVD}.
  Assume $x^\dagger \in \mathcal{N}(A)^\bot$ and let $\psi$ be an index function such that
\begin{equation} \label{eq:technical_source_condition}
\exists 
\bar{\beta} \in (0,1]: \forall \beta \in (0,\bar{\beta}): \sum_{j:\ \sigma_j^2 \leqslant \beta} \langle x^\dagger, v_j \rangle^2 = \mathcal{O}( \psi(\beta)^{2})
\end{equation}
holds. Then there exists another index function $\tilde{\psi}$ such that the local approximation property \eqref{eq:conv_trainingproperty_ass} is fulfilled for $x^\dagger$.
\end{lemma}

\begin{proof}
With the setting of Lemma \ref{lem:squared_soft_TSVD} and the trained parameters $(W,b)$ of the network $f_\theta$, it holds
\begin{equation}
    \|\varphi_{\theta(L)}(x^\dagger) - A x^\dagger\| = \|(Id - W)x^\dagger - b - A x^\dagger\| \leqslant \|(Id - W - A) x^\dagger\| + \|b\|.
\end{equation}
Now, we estimate the convergence order of both terms w.r.t $(1-L) \to 0$.

For the first part, we exploit the simple computations 
\begin{align}
    \|(Id - W - A) x^\dagger\|^2 &= \left\| \sum_j (1 - w_j - \sigma_j^2) \langle x^\dagger, v_j \rangle v_j \right\|^2 \notag\\
    &= \sum_j (1 - \sigma_j^2 - \min \{1-\sigma_j^2, L\})^2 \langle x^\dagger, v_j \rangle^2\notag\\
    &= \sum_j \max\{1 - L - \sigma_j^2, 0\}^2 \langle x^\dagger, v_j \rangle^2\notag\\
    &= \sum_{\sigma_j^2 \leqslant  1-L} (1 - L - \sigma_j^2)^2 \langle x^\dagger, v_j \rangle^2\notag\\
    &\leqslant  \sum_{\sigma_j^2 \leqslant  1-L} (1 - L)^2 \langle x^\dagger, v_j \rangle^2\notag\\
    &= (1 - L)^2 \sum_{\sigma_j^2 \leqslant  1-L} \langle x^\dagger, v_j \rangle^2.
\end{align}
Using assumption \eqref{eq:technical_source_condition}, we obtain
\begin{equation}
    (1-L)^2 \sum_{\sigma_j^2 \leqslant  1-L} \langle x^\dagger, v_j \rangle^2 = \mathcal{O} \left( (1-L)^2 \psi(1-L)^2 \right).
\end{equation}

Regarding the second part, it holds
\begin{align}
        \|b\|^2 &= \sum_j b_j^2
        = \sum_j \max\{ 0, 1-L - \sigma^2 \}^2 \mu_j^2 \notag\\
        &= \sum_{\sigma_j^2 \leqslant 1-L} (1-L-\sigma_j^2)^2 \mu_j^2
        \leqslant (1-L)^2 \sum_{\sigma_j^2 \leqslant 1-L} \mu_j^2.
\end{align}
Since the values $\mu_j$ form by definition (see Lemma \ref{lem:squared_soft_TSVD}) an $\ell^2$ sequence, there exists an index function $\tilde{\psi} \geqslant  \psi$ such that
\begin{equation}
    \sum_{\sigma_j^2 \leqslant 1-L} \mu_j^2 = \mathcal{O} \left(\tilde{\psi}(1-L)^2\right).
\end{equation}

Overall, we obtain
\begin{equation}
    \|(Id - W - A) x^\dagger\| + \|b\| = \mathcal{O} \left( (1-L)^2 \tilde{\psi}(1-L)^2 \right),
\end{equation}
thus, the local approximation property \eqref{eq:conv_trainingproperty_ass} is fulfilled.
\end{proof}


While \eqref{eq:technical_source_condition} is quite a technical source condition, standard source conditions of the form
\begin{equation}
    \exists w \in X, \mu > 0 \colon \quad x^\dagger =  A^\mu w
\end{equation}
imply it.
In this case, the index function $\psi$ is of the form $\psi(\beta) = \beta^\mu$. The proof of the following corollary reveals the exact relation between standard source conditions and \eqref{eq:technical_source_condition}.


\begin{corollary}\label{corr:soft_TSVD_source_cond}
 We assume that the setting of Lemma \ref{lem:squared_soft_TSVD} holds. Let $\origA $ be compact. Then, for any $x^\dagger \in \mathcal{N}(\origA )^\bot$ the local approximation property \eqref{eq:conv_trainingproperty_ass} is fulfilled (i.e., $S=\mathcal{N}(\origA )^\bot$).
\end{corollary}
\begin{proof}
Let $x^\dagger \in \mathcal{N}(\origA )^\bot$ be arbitrary. By construction, $A=\origA ^\ast \origA $ is self-adjoint and non-negative. As $A$ is also compact, we can deduce from \cite[Thm. 1]{mathe2008general}: For any $\varepsilon>0$ there exists an index function $\psi:[0,1]\to [0,\infty)$ such that $x^\dagger \in \{ x \in X \,| \, x=\psi(A)w, \|w\| \leqslant (1+\varepsilon) \|x^\dagger\| \}$.  

This implies that there exists a $w \in X$ such that $x^\dagger=\psi(A)w$. 
We thus obtain for $\beta >0$
\begin{align}
\sum_{\sigma_j^2 \leqslant  \beta} \langle x^\dagger, v_j \rangle^2 &= \sum_{\sigma_j^2 \leqslant  \beta} \langle \psi(A)w, v_j \rangle^2 = \sum_{\sigma_j^2 \leqslant  \beta} \psi(\sigma_j^2)^2 \langle w, v_j \rangle^2 \notag\\
&\leqslant  \sum_{\sigma_j^2 \leqslant  \beta} \psi(\beta)^2 \langle w, v_j \rangle^2 
\leqslant 
    \|w\|^2  \psi(\beta)^2 \leqslant 
    (1+ \varepsilon)^2 \|x^\dagger\|^2 \psi(\beta)^2.
\end{align}
By Lemma~\ref{lem:squared_soft_local_approx}, we get the desired local approximation property \eqref{eq:conv_trainingproperty_ass}. 
\end{proof}

The verification of regularization properties by Theorem~\ref{thm:convergence} in terms of the local approximation property is thus in line with the existing theory exploited in Lemma~\ref{lem:filter_regularization_with_bias}, and it further illustrates the character of the local approximation property combining the approximation capabilities of the residual network and a source condition on the desired solution.


\subsection{Linear network with ReLU activation}
In the following, we include selected nonlinearities as activation functions in a shallow network, allowing for an analytical investigation of the resulting reconstruction scheme. 
Here, we start with a ReLU activation, which requires a certain assumption on the training dataset depending on the chosen nonlinearity. For simplicity, we do not include a bias in the architecture, in contrast to the architecture from the last section.

\begin{lemma}
\label{lem:linear_with_relu}
    Let $(v_j, \sigma_j^2)_j$ be the eigenvectors and eigenvalues of $A$ and $\varphi_\theta = Id - f_\theta$ be an iResNet which solves \eqref{eq:training_minproblem} with
    \begin{itemize}
        \item[(i)] $f_\theta(x) = \phi(Wx)$, where $\theta = W$ and $\Theta = \{ W \in L(X) \, | \, \exists: (w_j)_{j\in \N}: W = \sum_{j \in \mathbb{N}} w_j \langle \cdot, v_j \rangle v_j \}$ and $\phi$ being the ReLU function w.r.t.\ the eigenvectors, i.e., $\phi(x) = \sum_{j \in \N} \max(0, \langle v_j, x \rangle ) v_j$,
        \item[(ii)] for every eigenvector $v_j$ of $A$, the training dataset $\{x^{(i)}\}_{i\in \{1,\hdots,N\}}$ contains at least one $x^{(i)}$ s.t.\ $\langle x^{(i)}, v_j \rangle > 0$.
    \end{itemize}
    Then, the solution of \eqref{eq:training_minproblem} is $W = \sum_{j \in \mathbb{N}} w_j \langle \cdot, v_j \rangle v_j$, $w_j = \min \{1- \sigma_j^2, L\}$ and \eqref{eq:data_dep_filterfunc} holds with 
    \begin{equation}
        r_L(\sigma^2, s) = 
        \begin{cases}
            \frac{1}{\max\{\sigma^2, 1-L\}} & \text{if } s \geqslant 0,\\
        1 & \text{if } s < 0.
        \end{cases}
    \end{equation}
%
\end{lemma}

The proof of the lemma can be found in Appendix~\ref{sec:proof_of_linear_with_relu}.
The obtained filter function is now characterized by a varying behavior depending on the actual measurement $y$. This is expressed via the variable $s$ which represents the coefficients $\langle z, v_i \rangle = \sigma_i \langle y, u_i \rangle$ (see \eqref{eq:data_dep_filterfunc}, \eqref{eq:data_dep_filterfunc_y}).
Whenever the coefficient is positive, the reconstruction scheme behaves like the squared soft TSVD without bias discussed in Section~\ref{sec:soft_TSVD}, i.e., they share the same filter function for those $x^\dagger \in \mathcal{N}(A)$. 
In all other cases, the reconstruction method does not change the coefficients of the data, i.e., it behaves like the identity. 
Due to the relationship to the squared soft TSVD, we can immediately specify those $x^\dagger$ fulfilling the local approximation property, i.e., for the ReLU network, we have
\begin{equation}
S=\{ x\in \mathcal{N}(\origA )^\bot \, | \, \forall j\in \N: \langle v_j,x\rangle \geqslant 0 \}.
\end{equation}
Thus, the nonlinearity in the network architecture introduces restricted approximation capabilities as well as convergence guarantees on a limited subset of $X$ only.

\subsection{Linear network with soft thresholding activation}

At last, we want to analyze a network with soft thresholding activation, e.g., known from LISTA \cite{gregor2010}. This function is known to promote sparsity since it shrinks all coefficients and sets those under a certain threshold to zero. That is why only training data
with sufficiently large coefficients matters for the result of the training, and the condition $|\langle x, v_j\rangle | > \frac{\alpha_j}{L}$ is crucial in the following lemma.

\begin{lemma}
\label{lem:arch_soft_thresh}
    Let $(v_j, \sigma_j^2)_j$ be the eigenvectors and eigenvalues of $A$, $\alpha=(\alpha_j)_j, \alpha_j  \geqslant  0$ and $\varphi_\theta = Id - f_\theta$ be an iResNet which solves \eqref{eq:training_minproblem} with
    \begin{itemize}
        \item[(i)] $f_\theta(x) = \phi_\alpha(Wx)$, where $\theta = W$, $\Theta = \{ W \in L(X) \, | \, \exists (w_j)_{j\in \N}:  W = \sum_{j \in \mathbb{N}} w_j \langle \cdot, v_j \rangle v_j \}$ and $\phi_\alpha$ is the soft thresholding function w.r.t.\ the eigenvectors, i.e.,\ $\phi_\alpha(x) = \sum_{j \in \N} \mathrm{sign}(\langle x, v_j \rangle) \max(0, |\langle x, v_j \rangle | -\alpha_j ) v_j$,
        \item[(ii)] the training data set is $\{x^{(i)}\}_{i \in \{1,\hdots,N\}}$, where for any $j\in \N: \exists i: \ |\langle x^{(i)}, v_j \rangle| > \frac{\alpha_j}{L}$. 
    \end{itemize}
    Then, a solution of \eqref{eq:training_minproblem} is $W = \sum_{j \in \mathbb{N}} w_j \langle \cdot , v_j \rangle v_j$, $w_j = \min\{\frac{\alpha_j}{p_{L,j}} + 1 - \sigma_j^2, L \}$, $p_{L,j} = \frac{\sum_{i\in I_j(L)} |\langle x^{(i)},v_j \rangle|^2 }{\sum_{i\in I_j(L)} |\langle x^{(i)},v_j \rangle| }$, $I_j(L)=\{ i \in \{1, ..., N\} | \ |\langle x^{(i)} , v_j \rangle | > \frac{\alpha_j}{L} \}$  and \eqref{eq:data_dep_filterfunc} holds with
    \begin{equation} \label{eq:filter_soft_thresh}
    r_L(\sigma_j^2, s) = 
    \begin{cases}
              \frac{1}{ \max\{\sigma_j^2 - \frac{\alpha_j}{p_{L,j}}, 1-L \}} \frac{|s| - \alpha_j}{|s|} & \text{if } |s| > \frac{\alpha_j}{ w_j}, \\
       1 & \text{if } |s| \leqslant  \frac{\alpha_j}{w_j}.
    \end{cases}
    \end{equation}
For singular values $\sigma_j^2 = 1$, $w_j$ is not uniquely determined.
\end{lemma}
 The proof of the lemma can be found in Appendix~\ref{app:lem:arch_soft_thresh}. It follows the same line of reasoning as in the previous sections but is more technical due to the effects of the nonlinearity.  

So far, the filter function $r_L$ in Lemma \ref{lem:arch_soft_thresh} is only defined on the discrete values $\sigma_j^2$ (and not continuous for $\sigma^2 \in [0,1]$) since it depends on the coefficients $p_{L,j}$, $\alpha_j$ and $w_j$. 
However, if we assume continuous extensions $p_L(\sigma^2)$ with $p_L(\sigma_j^2) = p_{L,j}$, $w_L(\sigma^2)$ with $w_L(\sigma_j^2) = w_j$, and $\alpha(\sigma^2)$ with $\alpha(\sigma_j^2)=\alpha_j$, the function $r_L = r_L(\sigma^2, s)$ also becomes continuous. The continuity at the point $|s| = \frac{\alpha(\sigma^2)}{w_L(\sigma^2)}$ is assured by
\begin{equation}
\frac{1}{ \max\{\sigma^2 - \frac{\alpha(\sigma^2)}{p_L(\sigma^2)}, 1-L \}} \frac{\frac{\alpha(\sigma^2)}{w_L(\sigma^2)} - \alpha(\sigma^2)}{\frac{\alpha(\sigma^2)}{w_L(\sigma^2)}}
= \frac{1-w_L(\sigma^2)}{ \max\{\sigma^2 - \frac{\alpha(\sigma^2)}{p_L(\sigma^2)}, 1-L \}} = 1 = r_L(\sigma^2, s).
\end{equation}

To be able to interpret the filter function, suitable values for $s$ representing the coefficients $\langle z, v_i \rangle = \sigma_i \langle y, u_i \rangle$ (see \eqref{eq:data_dep_filterfunc}, \eqref{eq:data_dep_filterfunc_y}) need to be considered. One reasonable option is to choose $s$ according to the data on which $\varphi_\theta$ has been trained. We thus consider an eigenvector $v_j$ scaled with the coefficient $p_{L,j}$.
Since $\varphi_\theta$ minimizes \eqref{eq:training_minproblem}, we expect $\varphi_\theta(p_{L,j} v_j) \approx p_{L, j} A v_j = p_{L,j} \sigma_j^2 v_j$ and $\varphi_\theta^{-1}(p_{L,j} \sigma_j^2 v_j) \approx p_{L, j} v_j $, respectively. Accordingly, we choose $|s|$ proportional to $p_L(\sigma^2) \sigma^2$, i.e.,\ $|s| = \gamma\, p_L(\sigma^2) \sigma^2$ for different values $\gamma > 0$. Hence, the case $\gamma=1$ corresponds to a test vector $z$ with coefficients $|\langle z, v_i \rangle| = |s| = p_L(\sigma_j^2) \sigma_j^2$, which perfectly fits to the training data. Analogously, the cases $\gamma < 1$ (and $\gamma > 1$, respectively) correspond to test data whose coefficients are smaller (or bigger, respectively) than the average coefficients of the training data.


For $\gamma = 1$, the filter function $r_L$ can be written in a form, which allows for an easier interpretation. It holds
\begin{equation} \label{eq:SFS_plt_func}
    r_L\left(\sigma^2, \pm \, p_L(\sigma^2) \sigma^2 \right) = 
    \begin{cases}
        1 & \text{if } \, \sigma^2 \leqslant \frac{\alpha(\sigma^2)}{L p_L(\sigma^2)}, \\
        \frac{1}{ 1-L} \left( 1 -\frac{\alpha(\sigma^2)}{p_L(\sigma^2) \sigma^2} \right) & \text{if } \, \frac{\alpha(\sigma^2)}{L p_L(\sigma^2)} < \sigma^2 \leqslant \frac{\alpha(\sigma^2)}{p_L(\sigma^2)} + 1 - L, \\
        \frac{1}{\sigma^2} & \text{if } \, \sigma^2 > \frac{\alpha(\sigma^2)}{p_L(\sigma^2)} + 1 - L.
    \end{cases}
\end{equation}
The derivation can be found in Appendix~\ref{sec:computation_nonlin_filterfunc}.
Note that the filter function depends especially on the quotient of $\alpha(\sigma^2)$ (soft thresholding parameter) and $p_L(\sigma^2)$ (property of training data). To visualize the influence, we depicted the graph in Figure \ref{fig:filter_function_STS} for two different (constant) values of $\frac{\alpha(\sigma^2)}{p_L(\sigma^2)}$. As can be seen, the graph of $\sigma^2 \mapsto \sigma^2 r(\sigma^2, p_L(\sigma^2) \sigma^2)$ has two kinks. The first one depends mainly on the choice of $\frac{\alpha(\sigma^2)}{p_L(\sigma^2)}$, the second one mainly on the choice of $L$.

For $\gamma \neq 1$, the filter functions cannot be written in a similarly compact and easy form as in \eqref{eq:SFS_plt_func}. Instead, we illustrate them in Figure \ref{fig:filter_function_STSgamma}. In contrast to the case of $\gamma=1$, the graph of $\sigma^2 \mapsto \sigma^2 r(\sigma^2, p_L(\sigma^2) \sigma^2$ is not equal to one for the larger values of $\sigma^2$.

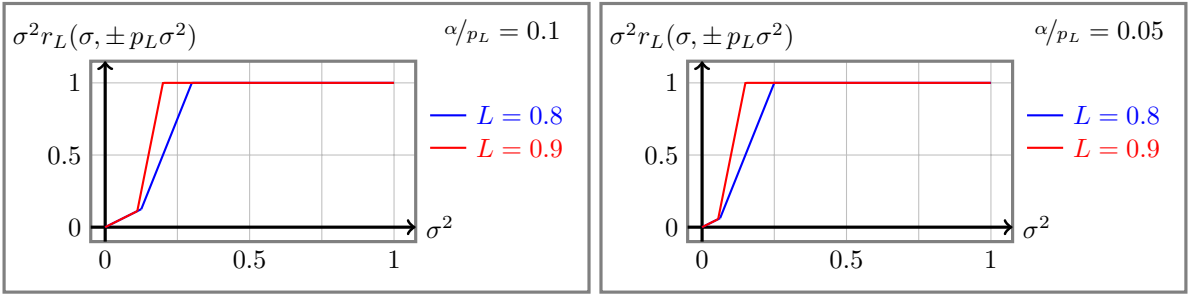
\begin{figure}[ht]
\centering
\begin{tikzpicture}[scale=0.96]
  \draw[step=1.0, lightgray, thin] (-0.2,-0.2) grid (4.3,2.3);
  \draw[gray, very thick] (-0.2,-0.2) rectangle (4.3,2.3);
  \draw[gray, very thick] (-1.4,-0.9) rectangle (6.7,3.1);
  \draw[->, thick, very thick] (-0.2, 0) -- (4.3, 0) node[right] {$\sigma^2$};
  \draw[->, thick, very thick] (0, -0.2) -- (0, 2.3) node[above] {$\sigma^2 r_L(\sigma, \pm \, p_L \sigma^2)$};
  \draw[scale=2, domain=0.25:2, variable=\t, blue, thick] plot[samples=161] ({\t}, {min(1,1/(1-0.8)*(0.5*\t-0.1))});
  \draw[scale=2, domain=0:0.25, variable=\t, blue, thick] plot[samples=21] ({\t}, {0.5*\t});
  \draw[scale=2, domain=0.222:2, variable=\t, red, thick] plot[samples=161] ({\t}, {min(1,1/(1-0.9)*(0.5*\t-0.1))});
  \draw[scale=2, domain=0:0.222, variable=\t, red, thick] plot[samples=21] ({\t}, {0.5*\t});
  
  \draw (-0.2,0) node[left]  {0};
  \draw (-0.2,1) node[left]  {0.5};
  \draw (-0.2,2) node[left]  {1};
  
  \node[below] at (0,-0.2) {$0$};
  \node[below] at (2,-0.2) {$0.5$};
  \node[below] at (4,-0.2) {$1$};
  
  \draw[blue, thick] (4.5,1.55) -- (5,1.55) node[right] {$L=0.8$};
  \draw[red, thick] (4.5,1.1) -- (5,1.1) node[right] {$L=0.9$};
  
  \draw (5.5,2.7) node {$\nicefrac{\alpha}{p_L}=0.1$};
\end{tikzpicture}
\begin{tikzpicture}[scale=0.96]
  \draw[step=1.0, lightgray, thin] (-0.2,-0.2) grid (4.3,2.3);
  \draw[gray, very thick] (-0.2,-0.2) rectangle (4.3,2.3);
  \draw[gray, very thick] (-1.4,-0.9) rectangle (6.7,3.1);
  \draw[->, thick, very thick] (-0.2, 0) -- (4.3, 0) node[right] {$\sigma^2$};
  \draw[->, thick, very thick] (0, -0.2) -- (0, 2.3) node[above] {$\sigma^2 r_L(\sigma, \pm \, p_L\sigma^2)$};
  \draw[scale=2, domain=0.125:2, variable=\t, blue, thick] plot[samples=161] ({\t}, {min(1,1/(1-0.8)*(0.5*\t-0.05))});
  \draw[scale=2, domain=0:0.125, variable=\t, blue, thick] plot[samples=21] ({\t}, {0.5*\t});
  \draw[scale=2, domain=0.111:2, variable=\t, red, thick] plot[samples=161] ({\t}, {min(1,1/(1-0.9)*(0.5*\t-0.05))});
  \draw[scale=2, domain=0:0.111, variable=\t, red, thick] plot[samples=21] ({\t}, {0.5*\t});
  
  \draw (-0.2,0) node[left]  {0};
  \draw (-0.2,1) node[left]  {0.5};
  \draw (-0.2,2) node[left]  {1};
  
  \node[below] at (0,-0.2) {$0$};
  \node[below] at (2,-0.2) {$0.5$};
  \node[below] at (4,-0.2) {$1$};
  
  \draw[blue, thick] (4.5,1.55) -- (5,1.55) node[right] {$L=0.8$};
  \draw[red, thick] (4.5,1.1) -- (5,1.1) node[right] {$L=0.9$};
  
  \draw (5.5,2.7) node {$\nicefrac{\alpha}{p_L}=0.05$};
\end{tikzpicture}

\caption{The plots depict graphs of $\sigma^2 r_L(\sigma^2, \pm \, p_L(\sigma^2) \sigma^2)$ as given in \eqref{eq:filter_soft_thresh} and \eqref{eq:SFS_plt_func}.
The choice $s=p_L(\sigma^2)\sigma^2$ corresponds to the case that the test data is similar to the average training data. In this case, $r_L$ shows a quite similar behavior as the squared soft TSVD function.} 
\label{fig:filter_function_STS}
\end{figure}

\begin{figure}[ht]
\centering
\begin{tikzpicture}[scale=0.96]
  \draw[step=1.0, lightgray, thin] (-0.2,-0.2) grid (4.3,2.5);
  \draw[gray, very thick] (-0.2,-0.2) rectangle (4.3,2.5);
  \draw[gray, very thick] (-1.5,-0.9) rectangle (6.7,3.3);
  \draw[->, thick, very thick] (-0.2, 0) -- (4.3, 0) node[right] {$\sigma^2$};
  \draw[->, thick, very thick] (0, -0.2) -- (0, 2.5) node[above] {$\sigma^2 r_L(\sigma, \pm \, \gamma p_L \sigma^2)$};
  \draw[scale=2, domain=0.29:2, variable=\t, blue, thick] plot[samples=161] ({\t}, {min(1 + (0.85 - 1)*0.1/(0.85*0.5*\t - 0.85*0.1),1/(1-0.8)*(0.5*\t-0.1/0.85))});
  \draw[scale=2, domain=0:0.29, variable=\t, blue, thick] plot[samples=21] ({\t}, {0.5*\t});
  \draw[scale=2, domain=0.26:2, variable=\t, red, thick] plot[samples=161] ({\t}, {min(1 + (0.85 - 1)*0.1/(0.85*0.5*\t - 0.85*0.1),1/(1-0.9)*(0.5*\t-0.1/0.85))});
  \draw[scale=2, domain=0:0.26, variable=\t, red, thick] plot[samples=21] ({\t}, {0.5*\t});
  
  \draw (-0.2,0) node[left]  {0};
  \draw (-0.2,1) node[left]  {0.5};
  \draw (-0.2,2) node[left]  {1};
  
  \node[below] at (0,-0.2) {$0$};
  \node[below] at (2,-0.2) {$0.5$};
  \node[below] at (4,-0.2) {$1$};
  
  \draw[blue, thick] (4.5,1.55) -- (5,1.55) node[right] {$L=0.8$};
  \draw[red, thick] (4.5,1.1) -- (5,1.1) node[right] {$L=0.9$};
  
  \draw (5.2,2.9) node {$\alpha = 0.1, p_L=1$};
  \draw (5.7,2.4) node {$\gamma = 0.85$};
\end{tikzpicture}
\begin{tikzpicture}[scale=0.96]
  \draw[step=1.0, lightgray, thin] (-0.2,-0.2) grid (4.3,2.5);
  \draw[gray, very thick] (-0.2,-0.2) rectangle (4.3,2.5);
  \draw[gray, very thick] (-1.5,-0.9) rectangle (6.7,3.3);
  \draw[->, thick, very thick] (-0.2, 0) -- (4.3, 0) node[right] {$\sigma^2$};
  \draw[->, thick, very thick] (0, -0.2) -- (0, 2.5) node[above] {$\sigma^2 r_L(\sigma, \pm \, \gamma p_L \sigma^2)$};
  \draw[scale=2, domain=0.22:2, variable=\t, blue, thick] plot[samples=161] ({\t}, {min(1 + (1.15 - 1)*0.1/(1.15*0.5*\t - 1.15*0.1),1/(1-0.8)*(0.5*\t-0.1/1.15))});
  \draw[scale=2, domain=0:0.22, variable=\t, blue, thick] plot[samples=21] ({\t}, {0.5*\t});
  \draw[scale=2, domain=0.25:2, variable=\t, red, thick] plot[samples=161] ({\t}, {min(1 + (1.15 - 1)*0.1/(1.15*0.5*\t - 1.15*0.1),1/(1-0.9)*(0.5*\t-0.1/1.15))});
  \draw[scale=2, domain=0.195:0.25, variable=\t, red, thick] plot[samples=161] ({\t}, {1/(1-0.9)*(0.5*\t-0.1/1.15)});
  \draw[scale=2, domain=0:0.195, variable=\t, red, thick] plot[samples=21] ({\t}, {0.5*\t});
  
  \draw (-0.2,0) node[left]  {0};
  \draw (-0.2,1) node[left]  {0.5};
  \draw (-0.2,2) node[left]  {1};
  
  \node[below] at (0,-0.2) {$0$};
  \node[below] at (2,-0.2) {$0.5$};
  \node[below] at (4,-0.2) {$1$};
  
  \draw[blue, thick] (4.5,1.55) -- (5,1.55) node[right] {$L=0.8$};
  \draw[red, thick] (4.5,1.1) -- (5,1.1) node[right] {$L=0.9$};
  
  \draw (5.2,2.9) node {$\alpha = 0.1, p_L=1$};
  \draw (5.7,2.4) node {$\gamma = 1.15$};
\end{tikzpicture}

\caption{The plots depict graphs of $\sigma^2 r(\sigma^2, \pm \, \gamma p_L(\sigma^2)\sigma^2)$ as given in \eqref{eq:filter_soft_thresh}. The case $\gamma \neq 1$ corresponds to test data, which differs from the average training data. Especially for larger singular values $\sigma^2$ (to the right of the two kinks), the filter function shows a suboptimal behavior compared to the squared soft TSVD filter function.}
\label{fig:filter_function_STSgamma}
\end{figure}
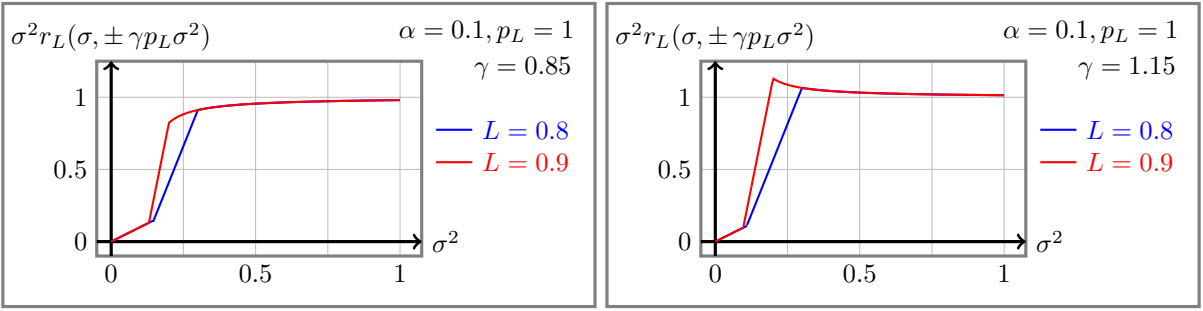

Finally, we want to analyze to which extent $\varphi_\theta$ fulfills the local approximation property \eqref{eq:conv_trainingproperty_ass}. The illustration of the filter function in Figure~\ref{fig:filter_function_STS} implies that convergence can only be possible if both kinks tend to zero. So, we need $L \to 1$ and $\alpha \to 0$.
Thus, the structure of the nonlinearity (soft thresholding) has a severe influence on the regularization properties of the reconstruction scheme.
But in contrast to the ReLU architecture, the soft thresholding operator provides the opportunity to control its similarity to a linear operator, which can be controlled via the coefficients $\alpha$. 
In the following lemma, we discuss in more detail how $\alpha$ can be chosen depending on the regularization parameter $L$ to obtain the desired regularization properties. 

\begin{lemma}
Let all assumptions of Lemma~\ref{lem:arch_soft_thresh} hold. Further, assume that $\origA $ is compact and let $p^\dagger=\sum_{j\in\N} p_j^\dagger v_j$ be given by
$p_j^\dagger = \frac{\sum_{i=1}^N |\langle x^{(i)},v_j \rangle|^2 }{\sum_{i=1}^N |\langle x^{(i)},v_j \rangle| }$.
In addition, consider $x^\dagger \in X$ as well as strictly monotonic and continuous architecture parameter choices $\alpha_j,\beta: (0,1) \rightarrow [0,\infty)$ with 

\begin{equation}\label{eq:STF_source_ass1}
  \alpha_j(L)\leqslant p_j^\dagger \beta(L),  \text{ with } \beta(L)= \mathcal{O}((1-L) \psi(1-L))   
\end{equation}
for an index function $\psi:[0,1]\to [0,\infty)$.

Then, the local approximation property \eqref{eq:conv_trainingproperty_ass} holds for any $x^\dagger \in \mathcal{N}(A)^\perp$.

\end{lemma}
\begin{proof}
We first verify that $p_{L},p^\dagger \in X$ ($p_L=\sum_{j\in \N} p_{L,j} v_j$).
For this, we exploit for each $j\in \N$ that 
\begin{equation}
\frac{|\langle x^{(i)},v_j\rangle |}{\sum_{i=1}^N |\langle x^{(i)},v_j \rangle|} \leqslant 1
\end{equation}
for any $i$. We thus obtain via Young's inequality
\begin{equation}
    \|p^\dagger\|^2 = \sum_{j\in\N} (p_j^\dagger)^2 \leqslant 
    \sum_{j\in\N} \bigg(\sum_{i=1}^N |\langle x^{(i)},v_j \rangle|\bigg)^2 \leqslant 2^{N-1} \sum_{i=1}^N \| x^{(i)}\|^2 < \infty.
\end{equation}
Analogously for $p_L$. Due to the finite nature of the data set and the properties of $\alpha$ we immediately have  $p_L \rightarrow p^\dagger$ for $L\to 1$ and $p_L=p^\dagger$ for $L$ being sufficiently close to $1$.

We now continue with the approximation property, making use of the notation $x_j:=\langle x^\dagger,v_j \rangle$ and a sufficiently large $L<1$ where $w_j=\min\{\frac{\alpha_j(L)}{p^\dagger_{j}} + 1 - \sigma_j^2, L \}\geqslant 0$:
\begin{align}
        \|\varphi_{\theta(L)}(x^\dagger) - A x^\dagger\|^2 & = \sum_{j\in \N} \left( \max(0,w_j |x_j| - \alpha_j(L))  - (1-\sigma_j^2) |x_j|  \right)^2 \notag \\
& \leqslant 2 \sum_{j\in \N} (1-\sigma_j^2-w_j)^2 |x_j|^2  + (w_j |x_j| - \max(0,w_j |x_j| - \alpha_j(L)) )^2 \notag \\ 
& = 2 \sum_{j\in \N} (1-\sigma_j^2-w_j)^2 |x_j|^2  + (\max(0, w_j |x_j|) - \max(0,w_j |x_j| - \alpha_j(L)) )^2 \notag \\ 
& \leqslant 2 \sum_{j\in \N} (1-\sigma_j^2-w_j)^2 |x_j|^2  + \alpha_j(L)^2 \notag \\ 
& \leqslant 2 \sum_{j\in \N} (1-\sigma_j^2-w_j)^2 |x_j|^2  + 2 \beta(L)^2 \|p^\dagger \|^2 \label{eq:soft_thres_source_auxil1}
        \end{align}
due to Young's inequality, the Lipschitz continuity of the ReLU function, and the assumption on $\alpha_j$. 
We further obtain
\begin{align}
        \sum_{j\in \N} (1-\sigma_j^2-w_j)^2 |x_j|^2  & =\underset{=:(I)}{ \underbrace{\sum_{j:  \sigma_j^2\leqslant 1-L + \frac{\alpha_j(L)}{p_j^\dagger}} (1-\sigma_j^2-w_j)^2 |x_j|^2}} \notag \\
 & + \underset{=:(II)}{ \underbrace{\sum_{j: \sigma_j^2> 1-L + \frac{\alpha_j(L)}{p_j^\dagger}}  (1-\sigma_j^2-w_j)^2 |x_j|^2 }} .      
        \end{align}

We thus need to derive estimates for $(I)$ and $(II)$.
\begin{align}
    (I) & = \sum_{j: \sigma_j^2\leqslant 1-L + \frac{\alpha_j(L)}{p_j^\dagger}} (1-\sigma_j^2-L)^2 |x_j|^2 \notag\\ 
    &\leqslant \sum_{j: \sigma_j^2\leqslant 1-L + \beta(L)} (1-\sigma_j^2-L)^2 |x_j|^2 \notag \\
        &\leqslant \sum_{j: \sigma_j^2\leqslant 1-L } (1-\sigma_j^2-L)^2 |x_j|^2 + \sum_{j:\sigma_j^2> 1-L  \land \sigma_j^2\leqslant 1-L + \beta(L)} (1-\sigma_j^2-L)^2 |x_j|^2 \notag \\
                &\leqslant \sum_{j: \sigma_j^2\leqslant 1-L } (1-L)^2 |x_j|^2 + \sum_{j:\sigma_j^2> 1-L  \land \sigma_j^2\leqslant 1-L + \beta(L)} \beta(L)^2 |x_j|^2 \notag \\
                &\leqslant (1-L)^2 \sum_{j: \sigma_j^2\leqslant 1-L } |x_j|^2 + \beta(L)^2 \|x^\dagger\|^2,  \label{eq:soft_thres_source_auxil2}
\end{align}
where we again exploit Young's inequality and the properties of $a_j$.
For $(II)$ we immediately obtain
\begin{equation}
    (II)= \sum_{j: \sigma_j^2> 1-L + \frac{\alpha_j(L)}{p_j^\dagger}} \left(\frac{\alpha_j(L)}{p^\dagger_j}\right)^2 |x_j|^2 \leqslant \beta(L)^2 \|x^\dagger \|^2. \label{eq:soft_thres_source_auxil3}
\end{equation}

Due to the properties of $\beta$, $(II)$ already has the desired characteristics for $L\to 1$. In contrast, $(I)$ requires some further calculations.
Following the same line of reasoning as in the proof of Corollary~\ref{corr:soft_TSVD_source_cond} we can find an index function $\psi'$ for any $\varepsilon>0$ such that
\begin{align}
    \sum_{j: \sigma_j^2\leqslant 1-L }   \langle x^\dagger, v_j \rangle^2 &  \leqslant  (1+\varepsilon)^2 \|x^\dagger\|^2  \psi'\left(1-L \right)^2 \notag \\
    & = \mathcal{O}(\psi'(1-L)^2).
\end{align}
Combining this with \eqref{eq:soft_thres_source_auxil1}, \eqref{eq:soft_thres_source_auxil2} and \eqref{eq:soft_thres_source_auxil3} we obtain the desired result.
\end{proof}

\begin{remark}
  Analogous to the line of reasoning in the proof of Lemma~\ref{lem:squared_soft_local_approx}, we split the series into two sums, (I) and (II). 
  (I) takes care of small singular values and needs to be related to a property of the element $x^\dagger$ in terms of a certain kind of source condition. The second sum (II) is somehow related to the approximation properties of the underlying architecture. In the proof of Lemma~\ref{lem:squared_soft_local_approx} it is immediately zero for any $x^\dagger$ which is due to the linear network assumption therein. In contrast, in the previous proof we had to carefully control the behavior of this term since it is strongly influenced by the structure of the nonlinearity.
\end{remark}

In general, the previous lemma serves as an example of the interplay between the approximation capabilities of the network and the resulting regularization properties. 

\section{Numerical experiments}
\label{sec:numerics}
In this section, we present numerical results in order to compare our theoretical findings from Section~\ref{sec:local_approx_property} and Section~\ref{sec:specialization} to its corresponding numerical implementations and extend these by experiments in a more general setting by learning from noisy measurements. To this end, we use a Radon transform with $30$ angles and $41$ detector pixels as our forward operator. We discretize the Radon transform for $28 \times 28$~px images. This results in a linear forward operator $\origA : \R^{28 \times 28} \to \R^{30 \times 41}$. For the training of the different iResNet architectures, we use the well-known MNIST dataset of handwritten digits~\cite{lecun1998mnist}. This dataset is split into \num{60000} images for training and \num{10000} images for testing. For our experiments, we treat the $28\times 28$ images as column vectors of length $784$ and use fully connected layers in all network architectures. We optimize the networks using Adam~\cite{KingmaB14} and a learning rate scheduling scheme.

There are multiple ways to satisfy the Lipschitz constraint during training. Bungert et al.~\cite{bungert2021clip} use an additional penalty term in the loss function. However, this method does not strictly enforce the constraint. Behrmann et al.~\cite{behrmann2019invertible} use contractive nonlinearities and only constrain the Lipschitz constant of each individual linear mapping. They compute the Lipschitz constant using a power iteration and normalize the weights after each training step. We observe an improvement of convergence when directly parameterizing the weights to fulfill a specific Lipschitz constant by extending the approach in~\cite{miyato2018spectral}.

In Section~\ref{sec:num_local_approximation}, we show experiments on the local approximation property in Theorem~\ref{thm:convergence}. To this end, we use the diagonal architecture proposed in Section~\ref{sec:diagonal_architecture}. The resulting data-dependent filter functions are visualized in Section~\ref{sec:num_data_dependent_filters}, and the convergence property is considered in Section~\ref{sec:num_convergence}. 

We implement each subnetwork $f_{\theta, j}, \ j=1, \dots, n$, (see \eqref{eq:diagonal_architecture_residual}) as a small fully-connected neural network with independent weights for each $j$, where $n$ denotes the total number of singular values in the discrete setting. Each subnetwork consists of three layers, where the first two layers $f_{\theta,j}^{k}$, $k = 1,2$ each contain $35$ hidden neurons, and its final layer $f_{\theta,j}^{3}$ contains $1$ neuron. Every layer is equipped with a linear matrix multiplication and corresponding additive bias and a ReLU activation function ($k=1,2$). Accordingly, each subnetwork has $1366$ trainable parameters. An illustration of the architecture is provided in Figure~\ref{fig:diagonal-numerics}. Altogether, the parameters of the subnetworks determine the parameters $\theta \in \Theta(L)$ of $\varphi_{\theta} = \mathrm{Id} - f_{\theta}$, where $\Theta(L)$ includes the network parameters as well as the Lipschitz constraint being realized by constraints on the network parameter. Here, we enforce the Lipschitz constants $\mathrm{Lip}(f_{\theta,j}) \leqslant L$, $j=1,\hdots,n$, by constraining $\mathrm{Lip}(f_{\theta,j}^{k}) \leqslant 1$ for $k = 1,2$ and $\mathrm{Lip}(f_{\theta,j}^{3}) \leqslant L$. In the following, we thus write $\theta(L)$ in order to give emphasis to the regularization parameter. Our training objective is the minimization of the \emph{approximation loss} \eqref{eq:training_minproblem}, i.e., minimize the loss
\begin{equation}
      l(\varphi_\theta,A)=  \frac{1}{N} \sum_{i} \|\varphi_\theta(x^{(i)}) - A x^{(i)} \|^2   
\end{equation}
subject to $\mathrm{Lip}(f_\theta) \leqslant  L$.

The source code corresponding to the experiments in this section is available at \url{https://gitlab.informatik.uni-bremen.de/inn4ip/iresnet-regularization}.


    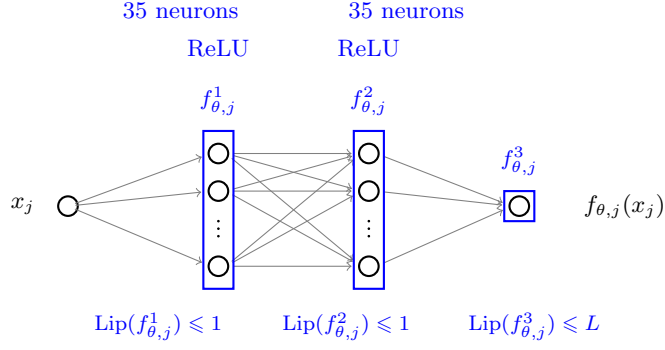
\begin{figure}
    \centering
        \begin{tikzpicture}[scale=1]
\small

\draw[color=black, thick] (0,4.3) circle (0.13);
\draw[->, color=gray,  shorten >=1.3mm, shorten <=1mm] (0,4.3) -- (1.9,5);
\draw[->, color=gray,  shorten >=1.3mm, shorten <=1mm] (0,4.3) -- (1.9,4.5);
\draw[->, color=gray,  shorten >=1.3mm, shorten <=1mm] (0,4.3) -- (1.9,3.5);
\draw[color=black, thick] (2,5) circle (0.13);
\draw[color=black, thick] (2,4.5) circle (0.13);
\filldraw[color=black] (2,4.1) circle (0.01);
\filldraw[color=black] (2,4) circle (0.01);
\filldraw[color=black] (2,3.9) circle (0.01);
\draw[color=black, thick] (2,3.5) circle (0.13);

\draw[->, color=gray,  shorten >=1.3mm, shorten <=1mm] (2.1,5) -- (3.9,5);
\draw[->, color=gray,  shorten >=1.3mm, shorten <=1mm] (2.1,5) -- (3.9,4.5);
\draw[->, color=gray,  shorten >=1.3mm, shorten <=1mm] (2.1,5) -- (3.9,3.5);
\draw[->, color=gray,  shorten >=1.3mm, shorten <=1mm] (2.1,4.5) -- (3.9,5);
\draw[->, color=gray,  shorten >=1.3mm, shorten <=1mm] (2.1,4.5) -- (3.9,4.5);
\draw[->, color=gray,  shorten >=1.3mm, shorten <=1mm] (2.1,4.5) -- (3.9,3.5);
\draw[->, color=gray,  shorten >=1.3mm, shorten <=1mm] (2.1,3.5) -- (3.9,5);
\draw[->, color=gray,  shorten >=1.3mm, shorten <=1mm] (2.1,3.5) -- (3.9,4.5);
\draw[->, color=gray,  shorten >=1.3mm, shorten <=1mm] (2.1,3.5) -- (3.9,3.5);

\draw[color=black, thick] (4,5) circle (0.13);
\draw[color=black, thick] (4,4.5) circle (0.13);
\filldraw[color=black] (4,4.1) circle (0.01);
\filldraw[color=black] (4,4) circle (0.01);
\filldraw[color=black] (4,3.9) circle (0.01);
\draw[color=black, thick] (4,3.5) circle (0.13);

\draw[->, color=gray,  shorten >=1.3mm, shorten <=1mm] (4.1,5) -- (5.9,4.3);
\draw[->, color=gray,  shorten >=1.3mm, shorten <=1mm] (4.1,4.5) -- (5.9,4.3);
\draw[->, color=gray,  shorten >=1.3mm, shorten <=1mm] (4.1,3.5) -- (5.9,4.3);
\draw[color=black, thick] (6,4.3) circle (0.13);

\draw[blue, thick]   (1.8,5.3) rectangle (2.2,3.2) ;
\draw[blue, thick] (3.8,5.3) rectangle (4.2,3.2);
\draw[blue, thick] (5.8,4.5) rectangle (6.2,4.1);

\draw[] (-0.6,4.3) node {$x_j$};

\draw[] (7.4,4.3) node {$f_{\theta,j}(x_j)$ };

\draw[blue] (2.0,5.7) node {\footnotesize$f_{\theta,j}^1$};
\draw[blue] (4.0,5.7) node {\footnotesize$f_{\theta,j}^2$};
\draw[blue] (6.0,4.9) node {\footnotesize$f_{\theta,j}^3$};

\draw[blue] (2.0,6.4) node {\small ReLU};
\draw[blue] (4.0,6.4) node {\small ReLU};

\draw[blue] (1.5,6.9) node {$35$ neurons};
\draw[blue] (4.5,6.9) node {$35$ neurons};


\draw[blue] (1.2,2.7) node {\footnotesize$\mathrm{Lip}(f_{\theta,j}^1) \leqslant 1$};
\draw[blue] (3.7,2.7) node {\footnotesize$\mathrm{Lip}(f_{\theta,j}^2) \leqslant 1$};
\draw[blue] (6.2,2.7) node {\footnotesize$\mathrm{Lip}(f_{\theta,j}^3) \leqslant L$};

\end{tikzpicture}
\caption{Illustration of the subnetworks $f_{\theta,j}:\R \to \R, \ j=1,\hdots,n$. Each having $1366$ trainable parameters.}
\label{fig:diagonal-numerics}
    \end{figure}

\subsection{Local approximation property}
\label{sec:num_local_approximation}

According to Theorem~\ref{thm:convergence}, we expect $\varphi_{\theta(L)}^{-1}$ to act as a regularizer as soon as the local approximation property~\eqref{eq:conv_trainingproperty_ass} is fulfilled. Note that the architecture does not change for varying $L$ in our experiments such that we write $\varphi_{\theta(L)}$ instead of $\varphi_{\theta(L),L}$. To be able to observe this behavior of our trained network, the local approximation property needs to be verified with respect to the trained network parameters. Therefore, we evaluate the approximation error with respect to fixed data samples
\begin{equation}
    \mathcal{E}_\mathrm{x^{(i)}}(\varphi_{\theta(L)},A) = \|\varphi_{\theta(L)}(x^{(i)}) - A x^{(i)} \|
\end{equation}
as well as the mean approximation error over the test data set
\begin{equation}
    \mathcal{E}_\mathrm{mean}(\varphi_{\theta(L)},A) = \frac{1}{N} \sum_{i} \|\varphi_{\theta(L)}(x^{(i)}) - A x^{(i)} \|.
\end{equation}
Figure~\ref{fig:loc_approx_MNIST} indicates superlinear convergence of $\mathcal{E}_\mathrm{mean}$ for $L \to 1$, i.e., the existence of an index function $\psi$ such that the local approximation property~\eqref{eq:conv_trainingproperty_ass} is satisfied within our test data set on average. This underpins the capability of the chosen network architecture and training to locally approximate the operator $A$ in terms of~\eqref{eq:conv_trainingproperty_ass}. Furthermore, the evaluation of $\mathcal{E}_{x^{(1)}}$ shows that the chosen sample $x^{(1)}$ fulfills the superlinear convergence and behaves very similarly to the mean error over the test data set. However, from the selection of data sample $x^{(2)}$ and corresponding error $\mathcal{E}_{x^{(2)}}$, we notice that the local approximation property does not hold for all samples of our test data set. Figure~\ref{fig:loc_approx_MNIST} additionally shows that some coefficients $x^{(2)}_j$ of $x^{(2)}$, corresponding to large singular values, severely deviate from the corresponding mean values with respect to the data distribution. This effect is reduced for $x^{(1)}$. Therefore, the slow, respectively non-, convergence of $\mathcal{E}_{x^{(2)}}$ could possibly be explained by the fact that structures outside of the core of the distribution of chosen singular values have not been learned properly during network training.


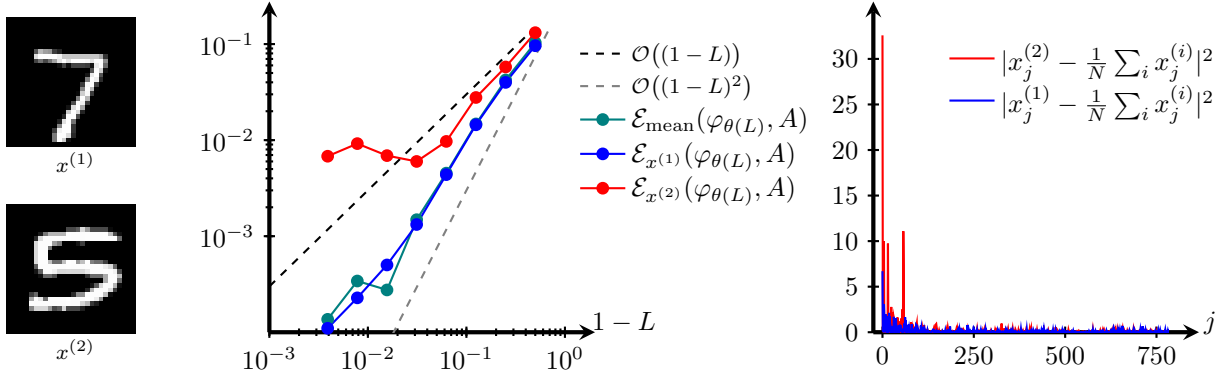
\begin{figure}
\captionsetup[subfigure]{labelformat=empty}
\begin{subfigure}{.13\textwidth}
    \centering

    \includegraphics[width=0.8\textwidth]{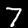}
    \vspace{-1.0ex}
    \caption{\footnotesize{$x^{(1)}$}}
    
    \vspace{2.0ex}
    \includegraphics[width=0.8\textwidth]{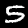}
    \vspace{-1.0ex}
    \caption{\footnotesize{$x^{(2)}$}}
    
    \vspace{1.0ex}
\end{subfigure}
\hfill
\begin{subfigure}{.85\textwidth}
    \centering
    \begin{tikzpicture}
        \pgfplotsset{
        every axis plot/.append style={thick},
        tick style={black, thick},
        every axis plot/.append style={line width=0.8pt},
        every axis/.style={
            axis y line=left,
            axis x line=bottom,
            axis line style={ultra thick,->,>=latex, shorten >=-.4cm}
        },
        }
        \begin{groupplot}[
            group style={
                group name=my plots,
                group size=2 by 1,
                ylabels at=edge left,
                horizontal sep=4.1cm
            },
            scale only axis=true,
            width=0.28\linewidth,
            height=4cm,
            axis lines=middle, 
            xlabel = {$1-L$},
            ytick = {0.,5,10,15,20,25,30},
            xtick = {0, 250,500, 750},
            ymin = 0.0,
            ymax = 0.22,
            xmin = 0.0,
             x label style={at={(axis description cs:1.14,0.1)},anchor=north},
             y label style={at={(axis description cs:0.035,1.32)},anchor=north},
        ]

        \nextgroupplot[
            xmode=log,ymode=log, 
            ymin =0.0001, ymax=0.135,
            xmin=0.001, xmax=1,
            axis y line=left,
            ytick={0.001,0.01,0.1},
            xtick={0.001,0.01,0.1,1},
            axis x line=bottom,
            x label style={at={(axis description cs:1.2,0.1)},anchor=north},
            y label style={at={(axis description cs:0.025,1.27)},anchor=north},
            legend style={at={(1.87,1)},draw=none},
            legend cell align={left}
        ]
        \addplot[color=black,dashed, domain=0.001:1]{0.3*x};
        \addlegendentry{{\footnotesize{$\mathcal{O}\big((1-L)\big)$}}}
        \addplot[color=gray,dashed, domain=0.001:1]{0.3*x^2};
        \addlegendentry{{\footnotesize{$\mathcal{O}\big((1-L)^2\big)$}}}
        \addplot[teal,mark=*] table [x=a, y=norm, col sep=comma]{csv_files/loc_approx_MNIST_updated_2.csv};
        \addlegendentry{{$\mathcal{E}_\mathrm{mean}(\varphi_{\theta(L)},A)$}}
        \addplot[blue,mark=*] table [x=a, y=xtrue_norm, col sep=comma]{csv_files/loc_approx_MNIST_updated_2.csv};
        \addlegendentry{{$\mathcal{E}_{x^{(1)}}(\varphi_{\theta(L)},A)$}}
        \addplot[red,mark=*] table [x=a, y=xfalse_norm, col sep=comma]{csv_files/loc_approx_MNIST_updated_2.csv};
        \addlegendentry{{$\mathcal{E}_{x^{(2)}}(\varphi_{\theta(L)},A)$}}

        \nextgroupplot[xlabel={$j$},ymin = 0, axis y line*=left,
            ymax = 33, xmin=-25,xmax=784, legend style={at={(1.2,1)},draw=none}]
        \addplot[red,mark=none] table [x=j, y=Ex2_squared, col sep=comma]{csv_files/dist_mean_x1_x2.csv};
        \addlegendentry{{$|x^{(2)}_{j} - \frac{1}{N} \sum_{i} x^{(i)}_j|^2$}}
        \addplot[blue,mark=none] table [x=j, y=Ex1_squared, col sep=comma]{csv_files/dist_mean_x1_x2.csv};
        \addlegendentry{{$|x^{(1)}_{j} - \frac{1}{N} \sum_{i} x^{(i)}_j|^2$}}
        \end{groupplot}
    \end{tikzpicture}
    \end{subfigure}
    \caption{Test samples $x^{(1)}$ and $x^{(2)}$~\textit{(left)}. Evaluations of $\mathcal{E}_\mathrm{mean}(\varphi_{\theta(L_m)},A), \ \mathcal{E}_{x^{(1)}}(\varphi_{\theta(L_m)},A)$ and $\mathcal{E}_{x^{(2)}}(\varphi_{\theta(L_m)},A)$ for $L_m = 1-\nicefrac{1}{2^m}$ with $m=1,\hdots,8$ on the MNIST test data set~\textit{(middle)}. 
    Squared absolute differences of the coefficients $x_j^{(k)}=\langle x^{(k)},v_j \rangle$ for $k=1,2$ to the mean value of corresponding coefficients in the test data set~\textit{(right)}.}
    
    \label{fig:loc_approx_MNIST}
\end{figure}


\subsection{Data-dependent filter functions for diagonal architecture}
\label{sec:num_data_dependent_filters}



For the experiments in this and the subsequent section, we train networks $\varphi_{\theta(L)}$ with Lipschitz bounds $L\in \{L_m=1-\nicefrac{1}{3^m} \, | \, m=1,\hdots5 \}$. We also include additive Gaussian noise in the network training via minimizing the approximation loss~\eqref{eq:training_minproblem}. More precisely, for each training sample $x^{(i)}$ we generate $Ax^{(i)}+\eta^{(i)}$ with $\eta^{(i)} \sim \mathcal{N}(0,\delta_\ell \mathrm{Id})$ and relative noise level ${\delta_\ell = \hat{\delta}_\ell \cdot \mathrm{std}_\mathrm{MNIST}}$, where ${\mathrm{std}_\mathrm{MNIST}}$ denotes the averaged standard deviation of the coefficients $\langle x^{(i)}, v_j \rangle$ of the MNIST data set (standard deviation with respect to $i$, mean with respect to $j$) and
\begin{equation}\label{eq:noise_levels}
    \hat{\delta}_\ell = \begin{cases}
        \left(\frac{1}{3}\right)^{7-\ell} &\text{ for } 0 < \ell < 7 \\
        0 &\text{ for } \ell = 0.
    \end{cases}
\end{equation} 
The loss function $l$ then amounts to 
\begin{equation}\label{eq:loss_function_approx}
   \min_{\theta \in \Theta(L)} l(\varphi_\theta,A) = \min_{\theta \in \Theta(L)}\frac{1}{N} \sum_{i=1}^N\| \varphi_{\theta}(x^{(i)})-Ax^{(i)}-\eta^{(i)} \|^2.
\end{equation}
Note that this includes~\eqref{eq:training_minproblem} in the noise-free case.
Trained networks on noisy samples with noise level $\delta$ with a particular Lipschitz bound $L$ are denoted by $\varphi_{\theta(L,\delta)}$. The noise-free training outcome is denoted by $\varphi_{\theta(L)}$.

Utilizing identical network architectures as in the previous subsection, we evaluate the learned filter functions of chosen networks. Analogously to \eqref{eq:linear_filter_with_bias} the reconstruction can be written as 
\begin{equation}
    T_L(y)= \varphi_\theta^{-1}(\origA ^\ast y) =  \hat{b}_L + \sum_{j \in \mathbb{N}} \hat{r}_L (\sigma_j^2, \sigma_j \langle y, u_j \rangle) \sigma_j \langle y, u_j \rangle v_j
\end{equation}
for $\hat{b}_L \in X=\mathbb{R}^n$.
As stated in Section~\ref{sec:diagonal_architecture}, the learned filter function of a trained network with diagonal architecture follows immediately from its subnetworks. Due to the additional bias, we make a small adaption to~\eqref{eq:filter_function}, which gives
\begin{equation}
    (\mathrm{Id}-f_{\theta,j})^{-1}(s) - \hat{b}_{L,j}  = r_L(\sigma_j^2,s)s \qquad \text{for} \ s \in \mathbb{R},
\end{equation}
where each entry $\hat{b}_{L,j}=\langle b_L, v_j\rangle$ corresponds to the axis intercept of a subnetwork $\varphi_{\theta,j}, \ j = 1,\hdots,n$, i.e., $\hat{b}_{L,j}=(\mathrm{Id}-f_{\theta,j})^{-1}(0)$. Since $f_{\theta,j}$ corresponds to $\sigma_j$ or $\sigma_j^2$, respectively, we also write $\hat{b}_{L,j} = \hat{b}_L(\sigma_j^2)$.

Adapting to the mean values with respect to the distribution of each coefficient, we compute 
\begin{equation}
    \mu_j := \frac{1}{N} \sum_{i=1}^N \langle x^{(i)},v_j \rangle \quad \text{for} \ j =1,\hdots,n
\end{equation} and evaluate the filter function with respect to its second variable at the continuous extension $\sigma^2\mu(\sigma^2)$ with $\mu(\sigma_j^2)=\mu_j$, $j=1,\hdots,n$. 

From Figure~\ref{fig:learned_filter_functions} we notice that the regularization of smaller singular values increases with decreasing Lipschitz bound $L$ for the mean MNIST data sample. This is in line with the regularization theory, as $L$ serves as our regularization parameter. Thus, the filter plots in Figure~\ref{fig:learned_filter_functions} visualize how $L$ acts as a regularization parameter. The trained diagonal networks $\varphi_{\theta(L_m)}$ for $L_2$ and $L_5$ show a slightly different behavior at $\sigma_j^2\approx 0.27$. The distributions of the coefficients $\langle x, v_j \rangle_{j=1,\hdots,n}$ are treated independently and in this particular singular value we observed a wider distribution of coefficients $\langle x, v_j \rangle$ in the dataset. The inherent structure within the dataset might be one possible explanation for this outlier. 
In general, when neglecting the previously mentioned outlier, for the mean MNIST sample one can observe a similarity to the squared soft TSVD filter function (see Section~\ref{sec:soft_TSVD}) to some extent. 
In addition, the observed decay of $\|\hat{b}_L\|$ with increased $L$ is also in line with the necessary condition for a classical filter function with bias to become a convergent regularization (see Lemma~\ref{lem:filter_regularization_with_bias}).
The observed increase for the largest $L$ is most likely caused by a numerical instability when evaluating $(\mathrm{Id}-f_{\theta,j})^{-1}(0)$ via the fixed point iteration performed for 30 iterations.

\begin{figure}
\centering
\begin{tikzpicture}
    \begin{groupplot}[
        group style={
            group name=my plots,
            group size=2 by 1,
            ylabels at=edge left,
            horizontal sep=1.8cm, vertical sep=1.5cm
        },
        footnotesize,
        width=0.44\linewidth,
        height=5cm,
        tickpos=left,
        y tick label style={/pgf/number format/fixed,
        /pgf/number format/1000 sep = \thinspace},
        x tick label style={/pgf/number format/fixed,
        /pgf/number format/1000 sep = \thinspace},
        ytick align=outside,
        xtick align=outside,
        enlarge x limits=false,
        xmin=0,xmax=1,
        xlabel={$\sigma^2$},
        ymin=-0.0, ymax=1.2,
        legend pos=south east,
        legend style={draw=none},
        legend cell align={left},
        grid = both,
        xtick={0,0.25,0.5,0.75,1},
        ytick={0,0.5,1},
        anchor=north,
        clip=true,
    ]

        \nextgroupplot[ylabel={$\sigma^2\hat{r}_L(\sigma^2,\sigma^2\mu(\sigma^2))$}, no markers]
        
        \addplot table [x=sigma, y=f, col sep=comma, mark=none]{csv_files/filter_n5_m0_meancoefs.csv};
        \addlegendentry{{$L_5=0.996$}}
        \addplot table [x=sigma, y=f, col sep=comma, mark=none]{csv_files/filter_n2_m0_meancoefs.csv};
        \addlegendentry{{$L_2=0.889$}}
        \addplot table [x=sigma, y=f, col sep=comma, mark=none]{csv_files/filter_n1_m0_meancoefs.csv};
        \addlegendentry{{$L_1=0.667$}}
        \addplot[color=black,only marks,mark=*,mark size=1pt] table [x=sigma, y=zeros, col sep=comma, mark=none]{csv_files/filter_n2_m0_meancoefs.csv};
        
        \nextgroupplot[ylabel={$\|\hat{b}_L\|$}, 
        xlabel ={$L$},
        ymin=0, ymax=1.4,
        ytick={0,0.5,1.0}, 
        xtick={0.7,0.8,0.9,1},
        xmin = 0.667, xmax=0.996]
        \addplot table [x=L, y=norm_b, col sep=comma, mark=*]{csv_files/filter_norm_b_m0.csv};
    \end{groupplot}
\end{tikzpicture}
\caption{Learned filter functions $r_L(\sigma^2, \sigma^2\mu(\sigma^2))$ of the iResNet with diagonal architecture and zero noise level $\delta_0$ for different Lipschitz constraints $L$~\textit{(left)} where the eigenvalues $\sigma_j^2$ are highlighted as black bullets on the x-axis. $\|\hat{b}_L\|$ corresponding to the axis intersection $\hat{b}_L=\sum_j \hat{b}_{L,j} v_j$, evaluated for $L_m$, $m=1,\hdots,5$~\textit{(right)}.}
\label{fig:learned_filter_functions}
\end{figure}
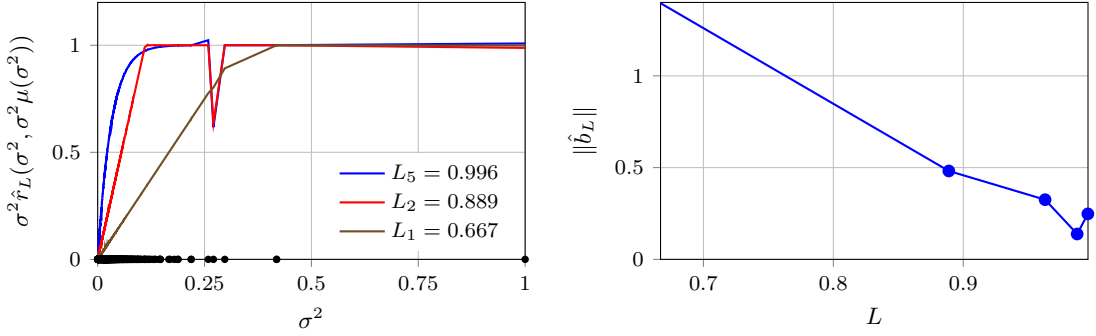

Figure~\ref{fig:reconstructions} includes reconstructions of a fixed sample from the test data set. This example illustrates the effect of a regularization induced by the Lipschitz bound even if the training takes noise into account. It becomes stronger for small values of $L$, which coincides with our previous observations. Reconstructions resulting from increased noise levels require stronger regularizations in order to improve reconstruction quality. Therefore, the best reconstructions in case of $\hat{\delta}_0$ and $\hat{\delta}_1$ result from $\varphi_{\theta(L_3,\delta_0)}^{-1}$ and $\varphi_{\theta(L_3,\delta_1)}^{-1}$. In comparison, $\varphi_{\theta(L_2,\delta_4)}^{-1}$ and $\varphi_{\theta(L_2,\delta_3)}^{-1}$ provide improved reconstructions for the cases $\hat{\delta}_4$ and $\hat{\delta}_3$. Moreover, at $L_1$ we notice similar blurred structures in the background of the reconstructed digit for all noise levels. One might argue that its structure compares to a handwritten digit itself, making the learned background pattern being encoded in the bias $\hat{b}_L$ suitable to the data set. These additional observations from Figure~\ref{fig:reconstructions} indicate a dependency of the regularization on learned data structures.
The corresponding filter functions being illustrated in the right column of Figure~\ref{fig:reconstructions} show a similar behavior for all training noise levels which underpins that the outcome of the approximation training is independent of the noise for a sufficiently large number of training samples.
The outlier in the filter function for the mean sample $(\mu_j)_j$ can also be observed in Figure~\ref{fig:reconstructions}. In addition, this test sample has a slightly different behavior with respect to the second singular value. Note that the seemingly linear behavior for $\sigma_j^2>0.4$ is only due to the fact that this is the linear continuous extension between the first and second singular value. 
In summary, the resulting filter behaves similarly to the squared soft TSVD filter independent of the training noise and with exceptions at two singular values.

\begin{figure}
    \centering
    \includegraphics[width=0.9\textwidth]{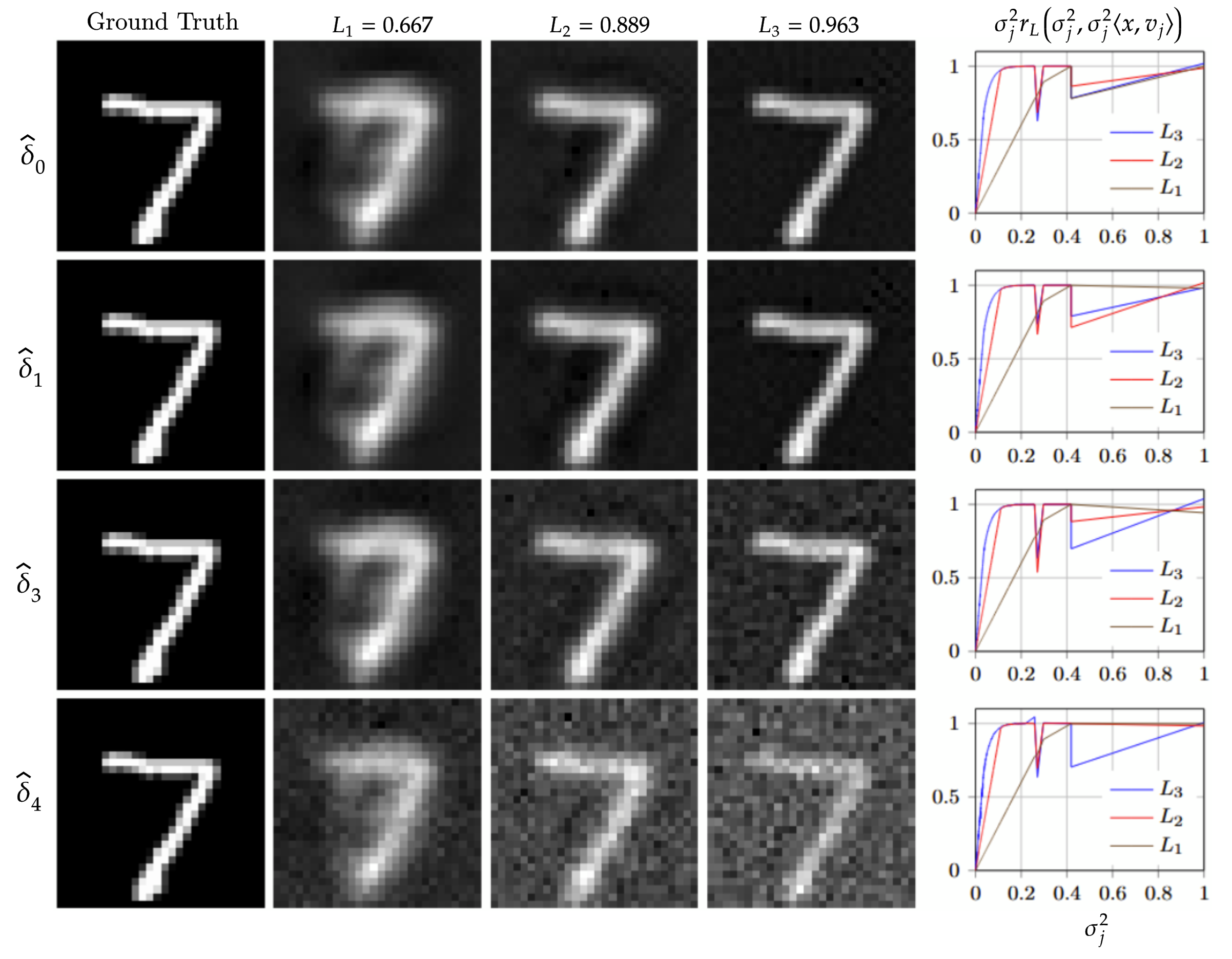}
    \caption{Reconstructions of an MNIST sample $\tilde{x}$ from the test data set by computing $\varphi_{\theta(L_m,\delta_\ell)}^{-1}(A\tilde{x}+\tilde{\eta})$ with $\tilde{\eta} \sim \mathcal{N}(0,\delta_\ell Id)$ for Lipschitz bounds $L_m$~\textit{(columns)} and noise levels $\delta_\ell= \hat{\delta}_\ell \cdot \mathrm{std}_\mathrm{MNIST}$ with $\ell =0,1,3,4$~\textit{(rows)}. The last column depicts the filter functions at $L_1$, $L_2$, and $L_3$ for each noise level with respect to the sample~$\tilde{x}$.}
    \label{fig:reconstructions}
\end{figure}

\subsection{Convergence property}
\label{sec:num_convergence}
After investigating approximation capabilities of trained networks with respect to the operator $A$ in Section~\ref{sec:num_local_approximation} in the noise-free case and extending the training to the noisy case in Section~\ref{sec:num_data_dependent_filters}, we continue verifying the convergence property with respect to different noise levels. 
We analyze the convergence property by computing the averaged \emph{reconstruction error}
\begin{equation}\label{eq:reco_error}
    \mathrm{MSE}_\mathrm{reco}^{\delta_\ell}(\varphi_{\theta(L,\delta)},A) = \frac{1}{N} \sum_{i=1}^N\| x^{(i)}-\varphi_{\theta(L,\delta)}^{-1}(Ax^{(i)}+\eta^{(i)}) \|^2,
\end{equation}
including the noise levels $\delta_\ell$ and noise samples $\eta^{(i)} \sim \mathcal{N}(0,\delta_\ell Id)$ in the reconstruction process where the network has been trained on noise level $\delta$. We thus can evaluate the noise-free training case, which solely aims to impart  data dependency, on reconstructions from noisy data, and the noisy training case where training and test noise share the same level. Reconstructions, i.e., the inverse of the iResNet, are obtained by using 30 fixed-point iterations. 

Figure~\ref{fig:MSE_reco} shows that $\varphi_{\theta(L,\delta_0)}^{-1}$ as well as $\varphi_{\theta(L,\delta_\ell)}^{-1}$ provides more accurate reconstructions with respect to ~\eqref{eq:reco_error} at large $L$ and low noise levels, whereas this behavior is reversed for decreasing Lipschitz bound and increasing noise levels. This is consistent with the regularization theory and the visualized reconstructions in Figure~\ref{fig:reconstructions}, as high noise levels require strong regularizations and vice versa. The behavior of $\mathrm{MSE}_\mathrm{reco}^{\delta_\ell}(\varphi_{\theta(L,\delta_\ell)},A)$ for small $L$ and small noise levels $\delta_\ell$ is rather intuitive, since its approximation capability is limited as a consequence of strong regularization. 
In addition, from Figure~\ref{fig:MSE_reco} one can also extract a suitable candidate for the parameter choice $L(\delta)$ to obtain convergence. 
The similarity in the behavior of $\varphi_{\theta(L,\delta_0)}^{-1}$ and $\varphi_{\theta(L,\delta_\ell)}^{-1}$ underpins that the outcome of the approximation training is independent of noise if data and noise are independent.

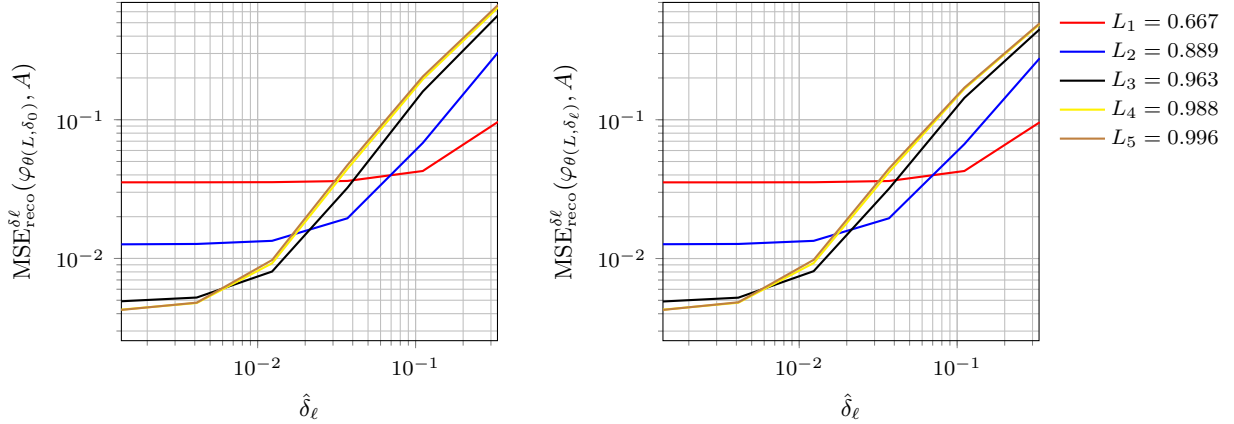
\begin{figure}
\centering
\begin{tikzpicture}
\pgfplotsset{
every axis plot/.append style={line width=0.8pt},
xmode = log,
ymode = log,
log y ticks with fixed point/.style={
      yticklabel={
        \pgfkeys{/pgf/fpu=true}
        \pgfmathparse{exp(\tick)}%
        \pgfmathprintnumber[fixed relative, precision=3]{\pgfmathresult}
        \pgfkeys{/pgf/fpu=false}}    
}
}
\begin{groupplot}[
    group style={
        group name=my plots,
        group size=2 by 1,
        ylabels at=edge left,
        xlabels at=edge bottom,
        horizontal sep=2.2cm
    },
    ymode=log,
    scale only axis,
    footnotesize,
    width=5cm,
    height=4.5cm,
    tickpos=left,
    x tick label style={/pgf/number format/fixed,
    /pgf/number format/1000 sep = \thinspace},
    ytick align=outside,
    xtick align=outside,
    enlarge x limits=false,
    xmin=0,xmax=0.33,
    xlabel={$\hat{\delta}_\ell$},
    legend pos=outer north east,
    legend style={draw=none},
    legend cell align={left},
    grid = both,
    xtick={0,0.0001,0.001,0.01,0.1,1},
    anchor = north,
    clip=true,
    cycle list name=color list
]
\nextgroupplot[ ylabel={$\mathrm{MSE}_\mathrm{reco}^{\delta\ell}(\varphi_{\theta(L,\delta_0)},A)$}, ymax = 0.7, ytick ={0.001,0.01,0.1,1}, yticklabels={$10^{-3}$,$ 10^{-2}$,$10^{-1}$}
 ]
\addplot table [x=noise, y=1, col sep=comma]{csv_files/MSEoverNoise_approx_reco_m=0_noise=True.csv};
\addplot table [x=noise, y=2, col sep=comma]{csv_files/MSEoverNoise_approx_reco_m=0_noise=True.csv};
\addplot table [x=noise, y=3, col sep=comma]{csv_files/MSEoverNoise_approx_reco_m=0_noise=True.csv};
\addplot table [x=noise, y=4, col sep=comma]{csv_files/MSEoverNoise_approx_reco_m=0_noise=True.csv};
\addplot table [x=noise, y=5, col sep=comma]{csv_files/MSEoverNoise_approx_reco_m=0_noise=True.csv};

 \nextgroupplot[ ylabel={$\mathrm{MSE}_\mathrm{reco}^{\delta\ell}(\varphi_{\theta(L,\delta_\ell)},A)$}, ymax = 0.7, ytick ={0.001,0.01,0.1,1}, yticklabels={$10^{-3}$,$ 10^{-2}$,$10^{-1}$}
 ]
\addplot table [x=noise, y=1, col sep=comma]{csv_files/MSEoverNoise_approx_reco_noise=True.csv};
\addlegendentry{$L_1 = 0.667$}
\addplot table [x=noise, y=2, col sep=comma]{csv_files/MSEoverNoise_approx_reco_noise=True.csv};
\addlegendentry{$L_2 = 0.889$}
\addplot table [x=noise, y=3, col sep=comma]{csv_files/MSEoverNoise_approx_reco_noise=True.csv};
\addlegendentry{$L_3 = 0.963$}
\addplot table [x=noise, y=4, col sep=comma]{csv_files/MSEoverNoise_approx_reco_noise=True.csv};
\addlegendentry{$L_4 = 0.988$}
\addplot table [x=noise, y=5, col sep=comma]{csv_files/MSEoverNoise_approx_reco_noise=True.csv};
\addlegendentry{$L_5 = 0.996$}
\end{groupplot}
\end{tikzpicture}
\caption{Reconstruction errors $\mathrm{MSE}_\mathrm{reco}^{\delta_\ell}(\varphi_{\theta(L,\delta_0)},A)$ with training on noise-free samples and reconstruction from noisy samples \textit{(left)} and $\mathrm{MSE}_\mathrm{reco}^{\delta_\ell}(\varphi_{\theta(L,\delta_\ell)},A)$ with training and reconstruction on the same noise level $\delta_\ell$~\textit{(right)} for $\varphi_{\theta(L_m,\delta_\ell)}$ over different noise levels $\delta_\ell=\hat{\delta}_\ell\cdot \mathrm{std}_\mathrm{MNIST}$, $\ell=0,\hdots,6$, and Lipschitz bounds $L_m$, $m=1,\hdots,5$.
}
\label{fig:MSE_reco}
\end{figure}

\section{Discussion and Outlook}
\label{sec:discussion}

In the present work, we developed and investigated the regularization theory for the proposed iResNet reconstruction approach providing a learned method from data samples. The network's local approximation property is fundamental to delivering a convergent regularization scheme. It comprises approximation properties of the architecture and training, a definition of solution type, and a source condition. 
The approximation loss used for training is motivated by this property. 
In the most general version, the framework can be understood as a fully-learned end-to-end reconstruction scheme with minor limitations as it relies on the concatenation with $\origA ^\ast$, i.e., some a priori knowledge on the forward operator is at least included in a hand-crafted way in the reconstruction scheme.
Introducing a diagonal architecture type relying on the SVD of the forward operator $\origA $ allowed for an analytical investigation of the resulting learned nonlinear spectral reconstruction method, which becomes a convergent regularization when fulfilling the local approximation property.
The analysis of trained shallow architectures revealed the link between the classical linear filter-based regularization theory and the concept of the local approximation property, and it illustrated the interplay between the approximation capability of the nonlinear network and the source condition. 
In addition, we validated and extended the theoretical findings by a series of numerical experiments on the MNIST data set and provided further insights into the learned nonlinear spectral regularization, such as similarity to the analytically determined linear regularization (squared soft TSVD) and data-dependency of the learned regularization.

Our iResNet method using the diagonal architecture can be seen as a generalization of the learned linear spectral regularization considered in \cite{kabri2022convergent}. Using a different loss, which measures the reconstruction error, the authors of \cite{kabri2022convergent} obtain learned filter functions corresponding to Tikhonov regularization with data- and noise-dependent regularization parameters for each singular vector. An extension of the iResNet approach to this kind of training loss is thus desirable for future comparison.

The present work also serves as a starting point for various future research directions:

As we obtain the inverse of the residual network via a fixed point iteration, findings in the context of the iResNet reconstruction may be related to learned optimization methods~\cite{banert2020data}, maximally monotone operators~\cite{pesquet2021learning}, or potentially to plug-and-play methods~\cite{ebner2022plug}. 
Investigations at the interface to these methods are far beyond the present work and will remain future work.

The concept of a local criterion to guarantee convergence of the iResNet reconstruction method provides further opportunities to investigate the regularization properties of learned reconstruction schemes. 
In a different context, the authors of \cite{aspri2020data} also consider a local property. They consider a learned projection onto subspaces concatenated with the forward operator, i.e., they consider a projected linear problem and combine it with Tikhonov regularization. There, the authors guarantee convergence of the regularization method by the Tikhonov theory, including a different but also localized assumption. 
The findings in the present work and the literature thus reveal the necessity to explicitly take the approximation capabilities of the learned schemes and the training procedure into account. 
Immediate future research includes a comprehensive numerical study on how architectural choices such as particular activation functions, e.g., being strictly monotone but nonlinear, influence the desired regularization properties and reconstruction quality. 
Further potential future investigations in this context are twofold. On the one hand, one may exploit universal approximation properties of the underlying architectures to guarantee convergence for a sufficiently large subset of $X$, e.g., by verifying the local approximation property. On the other hand, one may adapt the definition of regularization to the approximate and data-based nature of the learning-based setting. 

Besides the more general open research questions at the interface to other methods and data-based concepts, the iResNet reconstruction approach itself provides various potential future works. 
The observed similarity of the learned nonlinear filter functions in the numerical experiments to the analytical solution of the affine network with bias (Section~\ref{sec:soft_TSVD}) immediately raises the question: How much does the data's structure influence the resulting reconstruction scheme? Phrased differently, one needs to test the limits of the proposed approximation training. 
Here, it would be desirable to change the perspective to a Bayesian one considering data and noise distributions and further investigate the training outcome, including additional variations of the training loss and its influence on the local approximation property.
Natural extension and generalizations, such as including learned solution types, e.g., by combining our approach with null space networks \cite{schwab2019deep} or relaxing the diagonal assumption of the architecture, remain future work.

In summary, the present first work on learned iResNet regularization schemes builds the basis for various future works of theoretical and numerical nature.


%
%
%


\section*{Acknowledgments}
A. Denker, N. Heilenk\"otter, and M. Iske acknowledge the support by the Deutsche Forschungsgemeinschaft (DFG, German
Research Foundation) - Project number 281474342/GRK2224/2. T. Kluth acknowledges support from the DELETO project funded by the Federal Ministry of Education and Research (BMBF, project number 05M20LBB).




\appendix

\section{Proofs of Section \ref{sec:specialization}}

\subsection{Proof of Lemma \ref{lem:one_parameter_network}}
\label{sec:proof_of_one_parameter_network}

Equivalent filter function to the one-parameter-network.
\begin{proof}
    At first, we observe that the Lipschitz constraint in \eqref{eq:training_minproblem} is fulfilled for $|k| \leqslant  L$. Thus, \eqref{eq:training_minproblem} is equivalent to
    \begin{align}
     && &\min_{|k| \leqslant  L} \sum_{i} \|\varphi_\theta(x^{(i)}) - A x^{(i)} \|^2 \notag \\
     &\Leftrightarrow& &\min_{|k| \leqslant  L} \sum_{i} \|x^{(i)} - k x^{(i)} + k A x^{(i)} - A x^{(i)} \|^2 \notag  \\
     &\Leftrightarrow& &\min_{|k| \leqslant  L} \sum_{i} (1-k)^2 \|x^{(i)} - A x^{(i)} \|^2.
\end{align}
Since there is one $x^{(i)}$ s.t.\ $Ax_i \neq x^{(i)}$ the solution is obviously $k = L$ and we have $\varphi_\theta (x) = x - L (x - Ax)$.

Now we use the SVD of $\origA $ to solve 
$\varphi_\theta( \bar{x} ) = z$ for $z \in \mathcal{R}(A)$ component wise. For all $j \in \mathbb{N}$ it holds
\begin{align}
   && \langle \bar{x} - L (\bar{x} - A \bar{x}), v_j \rangle &= \langle z, v_j \rangle \notag \\
   &\Leftrightarrow& (1-L) \langle \bar{x}, v_j \rangle + L \langle \bar{x}, A v_j \rangle &= \langle z, v_j \rangle \notag  \\
   &\Leftrightarrow& (1-L + L \sigma_j^2) \langle \bar{x}, v_j \rangle &= \langle z, v_j \rangle \notag \\
   &\Leftrightarrow&  \langle \bar{x}, v_j \rangle &= \frac{1}{1-L + L \sigma_j^2} \langle z, v_j \rangle
\end{align}
Thus, the filter function
\begin{equation}
    r_L(\sigma^2, s) = \frac{1}{1-L + L \sigma^2}  = \frac{1}{L} \cdot \frac{1}{\alpha + \sigma^2}, \qquad \alpha = \frac{1-L}{L}
\end{equation}
fulfills \eqref{eq:data_dep_filterfunc}.
\end{proof}

\subsection{Proof of Lemma \ref{lem:squared_soft_TSVD}}
\label{sec:proof_of_squared_soft_TSVD}

Equivalent filter function to the linear network.
\begin{proof} 
    At first, we observe that the Lipschitz constraint in \eqref{eq:training_minproblem} is fulfilled if the eigenvalues $(w_j)_j$ of $W$ are restricted by $|w_j| \leqslant  L$. Defining $\Theta_L  = \left\lbrace (W, b) \in \Theta \, \big| \, |w_j| \leqslant L\right\rbrace$, \eqref{eq:training_minproblem} is equivalent to
    \begin{align}
     && &\min_{\theta \in \Theta_L} \, \sum_{i=1}^N \|\varphi_\theta(x^{(i)}) - A x^{(i)} \|^2\notag \\
     &\Leftrightarrow& &\min_{(W,b) \in \Theta_L} \, \sum_{i=1}^N \left\|x^{(i)} - Wx_i - b - A x^{(i)}  \right\|^2 \notag\\
     &\Leftrightarrow& &\min_{(W,b) \in \Theta_L} \, \sum_{i=1}^N \|  P_{\mathcal{N}(A)}(x^{(i)}) \|^2 + \left\| P_{\mathcal{N}(A)^\bot}(x^{(i)}) - Wx_i - b - A x^{(i)}  \right\|^2.
\end{align}
The part with $P_{\mathcal{N}(A)}(x^{(i)}) $ is independent of $W$ and $b$ and therefore not relevant. Furthermore, it holds
\begin{align}
    P_{\mathcal{N}(A)^\bot}(x^{(i)}) - Wx_i - b - A x^{(i)} &= \sum_{j \in \N} \langle x, v_j \rangle v_j - \sum_{j \in \N} \left( w_j \langle x, v_j \rangle + b_j \right) v_j - \sum_{j \in \N} \sigma_j^2 \langle x, v_j \rangle v_j \notag \\
    &= \sum_{j \in \N} \left( (1 - w_j - \sigma_j^2) \langle x, v_j\rangle -b_j \right) v_j.
\end{align}
Thus, the minimizing problem can be written as
\begin{equation}
     \min_{|w_j| \leqslant  L, \, (b_j) \in \ell^2} \, \sum_{i=1}^N  \sum_{j \in \mathbb{N}} \left( (1 - w_j - \sigma_j^2) \langle x, v_j\rangle -b_j \right)^2 .
\end{equation}

We can solve this problem for each $j$ separately. For determining $b_j$ we set the derivative to zero and get
\begin{align}
  & & 0 &\stackrel{!}{=} -2 \sum_{i=^1}^N \left( (1-w_j - \sigma_j^2) \langle x^{(i)}, v_j \rangle - b_j \right) \notag\\
    &\Leftrightarrow & b_j &\stackrel{!}{=} \frac{1}{N} (1-w_j - \sigma_j^2) \sum_{i=1}^N \langle x^{(i)}, v_j \rangle = (1-w_j - \sigma_j^2) \mu_j.
\end{align}
Since for every $i$, the sequence $(\langle x^{(i)}, v_j \rangle )_{j \in \N}$ is in $\ell^2$, $(b_j) \in \ell^2$ is also fulfilled.
It remains to solve
\begin{equation}
    \min_{|w_j| \leqslant  L,} \, \sum_{i=1}^N  \left( \left(1 - w_j - \sigma_j^2 \right) \left( \langle x, v_j\rangle - \mu_j \right) \right)^2.
\end{equation}
By assumption, for every $v_j$ there is at least one $x^{(i)}$ s.t.\ $\langle x^{(i)}, v_j \rangle \neq \mu_j$, thus, the problem can be simplified to
\begin{equation}
    \min_{|w_j| \leqslant  L} \, (1 - w_j - \sigma_j^2 )^2 \qquad \forall j \in \mathbb{N}.
\end{equation}
The obvious solution is $w_j = \min \{1- \sigma_j^2, L \}$, since $1-\sigma_j^2$ is by assumption always positive.

Now we use the SVD of $\origA $ to solve $\varphi_\theta( \bar{x} ) = z$ for some $z \in \mathcal{R}(A)$ component-wise.
For all $j \in \mathbb{N}$ it holds
\begin{align}
   && \langle \bar{x}, v_j \rangle - w_j \langle \bar{x}, v_j \rangle - b_j &= \langle z, v_j \rangle \notag\\
   &\Leftrightarrow& (1 - w_j) \langle \bar{x}, v_j \rangle &= \langle z, v_j \rangle  + b_j \notag\\
   &\Leftrightarrow& \langle \bar{x}, v_j \rangle &= \frac{1}{1 - w_j}\langle z, v_j \rangle + \frac{b_j}{1 - w_j} \notag\\
   &\Leftrightarrow& \langle \bar{x}, v_j \rangle &= \frac{1}{\max \{ \sigma_j^2, 1-L\} }\langle z, v_j \rangle + \frac{\max \{0, 1-L-\sigma_j^2\}}{\max \{ \sigma_j^2, 1-L\}} \mu_j.
\end{align}
Thus, using the filter function and bias
    \begin{equation}
        \hat{r}_L(\sigma^2) = \frac{1}{\max \{\sigma^2, 1-L \}}, \qquad \hat{b}_L = \sum_{\sigma_j^2 < 1-L} \frac{1-L-\sigma^2}{1-L} \mu_j v_j,
    \end{equation}
    the reconstruction schemes $\varphi_\theta^{-1} \circ \origA^*$ and $T_L$ from \eqref{eq:linear_filter_with_bias} are equivalent.
\end{proof}

\subsection{Proof of Lemma \ref{lem:linear_with_relu}}
\label{sec:proof_of_linear_with_relu}

Equivalent filter function to the ReLU-network.

\begin{proof}
    At first, we observe that the Lipschitz constraint in \eqref{eq:training_minproblem} is fulfilled if the eigenvalues $(w_j)_j$ of $W$ are restricted by $|w_j| \leqslant  L$. Thus, \eqref{eq:training_minproblem} is equivalent to
    \begin{align}
     && &\min_{(w_j), \,|w_j| \leqslant  L} \, \sum_{i} \|\varphi_\theta(x^{(i)}) - A x^{(i)} \|^2 \notag \\
     &\Leftrightarrow& &\min_{(w_j), \, |w_j| \leqslant  L} \, \sum_{i} \left\|x^{(i)} - \phi(Wx_i) - A x^{(i)}  \right\|^2 \notag \\
     &\Leftrightarrow& &\min_{(w_j), \, |w_j| \leqslant  L} \, \sum_{i} \left\| P_{\mathcal{N}(A)}(x^{(i)}) + \sum_{j \in \mathbb{N}} \langle x^{(i)}, v_j \rangle v_j - \sum_{j \in \mathbb{N}} \max\{0, w_j \langle x^{(i)}, v_j \rangle\} v_j - \sum_{j \in \mathbb{N}} \sigma_j^2 \langle x^{(i)}, v_j \rangle v_j  \right\|^2 \notag \\
     &\Leftrightarrow& &\min_{(w_j), \, |w_j| \leqslant  L} \, \sum_{i} \left\| P_{\mathcal{N}(A)}(x^{(i)}) \right\|^2 + \left\|\sum_{j \in \mathbb{N}} \left((1  - \sigma_j^2 )\langle x^{(i)}, v_j \rangle - \max\{0, w_j \langle x^{(i)}, v_j \rangle\} \right)v_j  \right\|^2 \notag\\
     &\Leftrightarrow& &\min_{(w_j), \, |w_j| \leqslant  L} \, \sum_{i}  \sum_{j \in \mathbb{N}} \left((1  - \sigma_j^2 )\langle x^{(i)}, v_j \rangle - \max\{0, w_j \langle x^{(i)}, v_j \rangle\} \right)^2 .
\end{align}
At first we focus on the components $\langle x^{(i)}, v_j \rangle$ which are negative or zero. They do not influence the solution of the problem since
\begin{equation}
    \min_{|w_j| \leqslant  L} \, \left( -(1 - \sigma_j^2) - \max\{0, -w_j\} \right)^2 \qquad \forall j \in \mathbb{N}
\end{equation}
is solved by any $w_j \in [0, L]$, as $1-\sigma_j^2$ is by assumption always positive. Thus, it suffices to consider the cases $\langle x^{(i)}, v_j \rangle > 0$, which lead to
\begin{equation}
    \min_{|w_j| \leqslant  L} \, (1 - \sigma_j^2 - \max\{0, w_j\} )^2 \qquad \forall j \in \mathbb{N}.
\end{equation}
The obvious solution is $w_j = \min \{1- \sigma_j^2, L \}$.

Now we use the SVD of $\origA $ to solve $\varphi_\theta( \bar{x} ) = z$ for $z \in \mathcal{R}(A)$ or
\begin{equation}
    \bar{x} - \sum_{j \in \mathbb{N}} \max \{0, w_j \langle \bar{x}, v_j \rangle \}v_j = \sum_{j \in \mathbb{N}} \langle z, v_j \rangle v_j,
\end{equation}
respectively, component-wisely.
For all $j \in \mathbb{N}$ it holds
\begin{align}
   && \langle \bar{x}, v_j \rangle -\max \{0, w_j \langle \bar{x}, v_j \rangle \} &= \langle z, v_j \rangle \notag \\
   &\Leftrightarrow& \min \{ \langle \bar{x}, v_j \rangle, (1-w_j)  \langle \bar{x}, v_j \rangle \}&=  \langle z, v_j \rangle.
\end{align}
Since $1-w_j>0$, the signs of $ \langle \bar{x}, v_j \rangle$ and $\langle z, v_j \rangle$ must coincide. In case of positive signs, the minimum on the left hand side takes the value $ (1-w_j)  \langle \bar{x}, v_j \rangle$, in case of negative signs, it is $\langle \bar{x}, v_j \rangle$. Thus,  we get
\begin{equation}
    \langle  \bar{x}, v_j \rangle  =
    \begin{cases}
        \frac{1}{1 - w_j} \langle z, v_j \rangle  & \text{if } \langle z, v_j \rangle \geqslant  0, \\
        \langle z, v_j \rangle & \text{if } \langle z, v_j \rangle < 0.
    \end{cases}
\end{equation}
Using $1-w_j = \max\{\sigma_j^2, 1-L\}$, it can be seen that
\begin{equation}
    r_L(\sigma^2, s) = 
    \begin{cases}
        \frac{1}{\max\{\sigma^2, 1-L\}} & \text{if } s \geqslant  0,\\
        1 & \text{if } s < 0
    \end{cases}
\end{equation}
fulfills \eqref{eq:data_dep_filterfunc}.
\end{proof}

\subsection{Proof of Lemma~\ref{lem:arch_soft_thresh}}
\label{app:lem:arch_soft_thresh}

Equivalent filter function to the soft-thresholding-network.
\begin{proof}
    At first, we observe that the Lipschitz constraint in \eqref{eq:training_minproblem} is fulfilled if the eigenvalues $(w_j)_j$ of $W$ are restricted by $|w_j| \leqslant  L$. 
    Besides, $f_\theta(x)$ can be written as
    \begin{align}
        \phi_\alpha(Wx) &= \sum_{j\in \N} \mathrm{sign}(\langle W x, v_j \rangle) \max(0, |\langle Wx, v_j \rangle | - \alpha_j ) v_j \notag\\
        &= \sum_{j\in \N} \mathrm{sign}(w_j \langle x, v_j \rangle) \max(0, |w_j \langle x, v_j \rangle | - \alpha_j ) v_j.
    \end{align}
    For any $x^{(i)} $ from the training data set, it follows
    \begin{align}
        & \quad \, \, x^{(i)} - \phi_\alpha(Wx^{(i)}) - A x^{(i)}  \notag\\
        &= P_{\mathcal{N}(A)}(x^{(i)}) + \sum_{j \in \mathbb{N}} \langle x^{(i)}, v_j \rangle v_j - \sum_{j\in \N} \mathrm{sign}(w_j \langle x^{(i)}, v_j \rangle) \max(0, |w_j \langle x^{(i)}, v_j \rangle | - \alpha_j ) v_j - \sum_{j \in \mathbb{N}} \sigma_j^2 \langle x^{(i)}, v_j \rangle v_j \notag\\
        &= P_{\mathcal{N}(A)}(x^{(i)}) + \sum_{j \in \mathbb{N}} \left[ x_j^{(i)}  -  \mathrm{sign}(w_j x_j^{(i)}) \max(0, |w_j x_j^{(i)} | - \alpha_j ) -  \sigma_j^2 x_j^{(i)} \right] v_j 
    \end{align}
    where $x_j^{(i)}:=\langle x^{(i)},v_j\rangle$.
Therefore,  \eqref{eq:training_minproblem} is equivalent to
    \begin{align}
     && &\min_{(w_j), \,|w_j| \leqslant  L} \, \sum_{i} \|\varphi_\theta(x^{(i)}) - A x^{(i)} \|^2 \notag \\
     &\Leftrightarrow& &\min_{(w_j), \, |w_j| \leqslant  L} \, \sum_{i} \left\|x^{(i)} - \phi_\alpha(Wx^{(i)}) - A x^{(i)}  \right\|^2 \notag \\
     &\Leftrightarrow& &\min_{(w_j), \, |w_j| \leqslant  L} \, \sum_{j}  \sum_{i}  \left( x_j^{(i)}  -  \mathrm{sign}(w_j x_j^{(i)}) \max(0, |w_j x_j^{(i)} | - \alpha_j ) -  \sigma_j^2 x_j^{(i)} \right)^2.
\end{align}
For each $j$, we have to find the solution of
\begin{align}
    && &\min_{|w_j| \leqslant  L}  \sum_{i}  \left( x_j^{(i)}  -  \mathrm{sign}(w_j x_j^{(i)}) \max(0, |w_j x_j^{(i)} | - \alpha_j ) -  \sigma_j^2 x_j^{(i)} \right)^2 \notag \\
    &\Leftrightarrow& &\min_{|w_j| \leqslant  L} \sum_{i} \left( (1- \sigma_j^2) \mathrm{sign}(x_j^{(i)}) |x_j^{(i)}| - \mathrm{sign}(w_j) \mathrm{sign}(x_j^{(i)}) \max(0, |w_j x_j^{(i)} | - \alpha_j )\right)^2 \notag \\
    &\Leftrightarrow& &\min_{|w_j| \leqslant  L} \sum_{i} \left( (1- \sigma_j^2) |x_j^{(i)}| - \mathrm{sign}(w_j) \max(0, |w_j x_j^{(i)} | - \alpha_j )\right)^2. \label{eq:soft_thr_auxil1}
\end{align}
Note that for $\sigma_j^2 = 1$, any $w_j \in \left[ -\frac{\alpha_j}{\max^{(i)}|x_j^{(i)}|}, \frac{\alpha_j}{\max_i|x_j^{(i)}|} \right]$ solves the problem and hence the solution is not unique. In all other cases, we have $1 - \sigma_j^2 > 0$, thus, only a positive sign of $w_j$ makes sense and we restrict ourselves to this case.
Due to $w_j \geqslant  0$, we have
\begin{equation}
    \quad \, \, \left( (1- \sigma_j^2) |x_j^{(i)}| -  \max(0, |w_j x_j^{(i)} | - \alpha_j )\right)^2 
    = 
    \begin{cases}
        \left((1- \sigma^2) x_j^{(i)}\right)^2 & \text{if } |w_j x_j^{(i)}| \leqslant  \alpha_j, \\
        \left((1- \sigma_j^2 - w_j) |x_j^{(i)}|  + \alpha_j \right)^2 & \text{if } |w_j x_j^{(i)}| > \alpha_j.
    \end{cases}
\end{equation}
As the first case is an upper bound of the second case, it is highly desirable to choose $w_j$ large enough such that all training samples are contained in the second case. But we also need to take into account the upper bound $w_j\leqslant L$. 
Due to the data set assumption (ii), the minimization problem \eqref{eq:soft_thr_auxil1} becomes
\begin{equation}
    \min_{w_j \in [0, L]} \sum_{i\in I_j(L)} \left( (1- \sigma_j^2 - w_j) |x_j^{(i)}| +\alpha_j \right)^2.
\end{equation}

The minimizer is
\begin{equation} \label{eq:wj_soft_thresholding}
    w_j = 
    \begin{cases}
        \frac{\alpha_j}{p_j} + 1 - \sigma_j^2 & \text{if } \frac{\alpha_j}{p_j} + 1 - \sigma_j^2 \leqslant  L, \\
        L & \text{else},
    \end{cases}
\end{equation}
where $p_{L,j}=p_{j} = \frac{\sum_{i\in I_j(L)} |x_j^{(i)}|^2 }{\sum_{i\in I_j(L)} |x_j^{(i)}| }$.
Equivalently, we can write $w_j = \min\{\frac{\alpha_j}{p_j} + 1 - \sigma_j^2, L \}$. 


Now we use the SVD of $\origA $ to solve $\varphi_\theta( \bar{x} ) = z$ for $z \in \mathcal{R}(A)$ or
\begin{equation}
    \bar{x} - \sum_{j\in \N} \mathrm{sign}(w_j \langle \bar{x}, v_j \rangle) \max(0, |w_j \langle \bar{x}, v_j \rangle | - \alpha_j ) v_j = \sum_{j \in \N} \langle z, v_j \rangle v_j,
\end{equation}
respectively, componentwise.
For all $j \in \mathbb{N}$ it holds
\begin{align}
   && \langle \bar{x}, v_j \rangle - \mathrm{sign}(w_j \langle \bar{x}, v_j \rangle) \max(0, |w_j \langle \bar{x}, v_j \rangle | - \alpha_j ) &= \langle z, v_j \rangle \notag \\
   &\Leftrightarrow&  \mathrm{sign}(\langle \bar{x}, v_j \rangle) \left( |\langle \bar{x}, v_j \rangle| - \max(0, w_j| \langle \bar{x}, v_j \rangle | - \alpha_j ) \right) &= \langle  z, v_j \rangle.
\end{align}
Thus, the sign of $\langle \bar{x}, v_j \rangle$ is the same as the one of $\langle z, v_j \rangle$. Assuming $w_j| \langle \bar{x}, v_j \rangle | \leqslant  \alpha_j$, we obtain
\begin{equation}
    |\langle \bar{x}, v_j \rangle| =  |\langle z, v_j \rangle| \qquad \text{if } |\langle z, v_j \rangle| \leqslant  \frac{\alpha_j}{w_j}
\end{equation}
and assuming $w_j| \langle \bar{x}, v_j \rangle | > \alpha_j$, we obtain
\begin{equation}
    |\langle \bar{x}, v_j \rangle| = \frac{ |\langle z, v_j \rangle| - \alpha_j}{1-w_j} \qquad \text{if }  \frac{ |\langle z, v_j \rangle| }{1-w_j} - \frac{\alpha_j}{1-w_j} > \frac{\alpha_j}{w_j}.
\end{equation}
In total, this leads to
\begin{equation}
    \langle \bar{x}, v_j \rangle = 
    \begin{cases}
        \langle z, v_j \rangle & \text{if } |\langle z, v_j \rangle| \leqslant  \frac{\alpha_j}{w_j},\\
        \mathrm{sign}(\langle z, v_j \rangle)\frac{|\langle z, v_j \rangle| - \alpha_j}{1-w_j} & \text{if } |\langle z, v_j \rangle| > \frac{\alpha_j}{w_j}.
    \end{cases}
\end{equation}
Using $1 - w_j = \max\{\sigma_j^2 - \frac{\alpha_j}{p_j}, 1-L \}$, we see that the filter function
\begin{equation}
    r_L(\sigma_j^2, s) = 
    \begin{cases}
       1 & \text{if } |s| \leqslant  \frac{\alpha_j}{w_j},\\
       \frac{1}{ \max\{\sigma_j^2 - \frac{\alpha_j}{p_j}, 1-L \}} \frac{|s| - \alpha_j}{|s|} & \text{if } |s| > \frac{\alpha_j}{ w_j}
    \end{cases}
\end{equation}
fulfills \eqref{eq:data_dep_filterfunc}. $F_L$ is then derived by exploiting \eqref{eq:data_dep_filterfunc_original_problem}.
\end{proof}

\subsection{Computation of filter function \eqref{eq:SFS_plt_func}}
\label{sec:computation_nonlin_filterfunc}

Assuming continuous extensions $p_L(\sigma^2)$ s.t.\ $p_L(\sigma_j^2) = p_{L,j}$, $w_L(\sigma^2)$ s.t.\ $w_L(\sigma_j^2) = w_j$, and $\alpha(\sigma^2)$ s.t.\ $\alpha(\sigma_j^2)=\alpha_j$, we can write the filter function from Lemma \ref{lem:arch_soft_thresh} as
\begin{equation}
    r_L(\sigma^2, s) = 
    \begin{cases}
    \frac{1}{ \max\{\sigma^2 - \frac{\alpha(\sigma^2)}{p_L(\sigma^2)}, 1-L \}} \frac{|s| - \alpha(\sigma^2)}{|s|} & \text{if } |s| > \frac{\alpha(\sigma^2)}{ w_L(\sigma^2)}, \\
    1 & \text{if } |s| \leqslant  \frac{\alpha(\sigma^2)}{w_L(\sigma^2)}.
    \end{cases}
\end{equation}
Besides, it holds $w_L(\sigma^2) = \min \left\lbrace \frac{\alpha(\sigma^2)}{p_L(\sigma^2)} + 1 - \sigma^2, L \right\rbrace \leqslant L$, also by Lemma \ref{lem:arch_soft_thresh}.
Now, we choose $|s| = p_L(\sigma^2) \sigma^2$ and distinguish between three different cases (small, medium and large values of $\sigma^2$).

At first, we consider the case $\sigma^2 \leqslant \frac{\alpha(\sigma^2)}{L p_L(\sigma^2)}$. Then, it holds
\begin{equation}
    |s| = p_L(\sigma^2) \sigma^2 \leqslant \frac{\alpha(\sigma^2)}{L} \leqslant \frac{\alpha(\sigma^2)}{w_L(\sigma^2)}.
\end{equation}
Thus, it follows
\begin{equation}
    r_L\left(\sigma^2, \pm \, p_L(\sigma^2) \sigma^2 \right) = 1.
\end{equation}

Secondly, we consider the case $\frac{\alpha(\sigma^2)}{L p_L(\sigma^2)} < \sigma^2 \leqslant \frac{\alpha(\sigma^2)}{p_L(\sigma^2)} + 1 - L$. We shortly make sure, that the lower bound is in fact smaller than the upper one. The assumptions of Lemma \ref{lem:arch_soft_thresh} imply that $p_L(\sigma^2) > \frac{\alpha(\sigma^2)}{L}$ holds. From this, we can follow
\begin{align}
    & & p_L(\sigma^2) (1-L) &> \frac{\alpha(\sigma^2)}{L} -  \frac{L\alpha(\sigma^2)}{L} \notag\\
    &\Rightarrow & \alpha(\sigma^2) + p_L(\sigma^2) (1-L) &> \frac{\alpha(\sigma^2)}{L} \notag\\
    &\Rightarrow & \frac{\alpha(\sigma^2)}{p_L(\sigma^2)} + 1 - L &> \frac{\alpha(\sigma^2)}{L p_L(\sigma^2)}.
\end{align}
Besides, note that $\sigma^2 \leqslant \frac{\alpha(\sigma^2)}{p_L(\sigma^2)} + 1 - L$ implies $w_L(\sigma^2) = L$. Thus, it holds
\begin{equation}
    |s| = p_L(\sigma^2) \sigma^2 > \frac{\alpha(\sigma^2)}{L} = \frac{\alpha(\sigma^2)}{w_L(\sigma^2)} .
\end{equation}
Accordingly, we obtain
\begin{align}
        r_L\left(\sigma^2, \pm \, p_L(\sigma^2) \sigma^2 \right) &= \frac{1}{ \max\{\sigma^2 - \frac{\alpha(\sigma^2)}{p_L(\sigma^2)}, 1-L \}} \frac{p_L(\sigma^2) \sigma^2 - \alpha(\sigma^2)}{p_L(\sigma^2) \sigma^2} \notag \\
        &= \frac{1}{ 1-L} \frac{p_L(\sigma^2) \sigma^2 - \alpha(\sigma^2)}{p_L(\sigma^2) \sigma^2} = \frac{1}{ 1-L} \left( 1 -\frac{\alpha(\sigma^2)}{p_L(\sigma^2) \sigma^2} \right).
\end{align}

At last, we consider the case $\sigma^2 > \frac{\alpha(\sigma^2)}{p_L(\sigma^2)} + 1 - L$. Note that this implies $w_L(\sigma^2) = \frac{\alpha(\sigma^2)}{p_L(\sigma^2)} + 1 - \sigma^2$. Besides it holds $\sigma^2 - \frac{\alpha(\sigma^2)}{p_L(\sigma^2)} > 0$, which implies
\begin{align}
    & & \left(\sigma^2 - \frac{\alpha(\sigma^2)}{p_L(\sigma^2)} \right) (1-\sigma^2) &> 0 \notag\\
    &\Rightarrow & \sigma^2 \left( \frac{\alpha(\sigma^2)}{p_L(\sigma^2)} + 1 - \sigma^2 \right) &> \frac{\alpha(\sigma^2)}{p_L(\sigma^2)} \notag\\
    &\Rightarrow & |s| = p_L(\sigma^2) \sigma^2 &> \frac{\alpha(\sigma^2)}{\frac{\alpha(\sigma^2)}{p_L(\sigma^2)} + 1 - \sigma^2 } = \frac{\alpha(\sigma^2)}{w_L(\sigma^2)}.
\end{align}
Accordingly, we obtain
\begin{align}
        r_L\left(\sigma^2, \pm \, p_L(\sigma^2) \sigma^2 \right) &= \frac{1}{ \max\{\sigma^2 - \frac{\alpha(\sigma^2)}{p_L(\sigma^2)}, 1-L \}} \frac{p_L(\sigma^2) \sigma^2 - \alpha(\sigma^2)}{p_L(\sigma^2) \sigma^2} \notag \\
        &= \frac{1}{\sigma^2 - \frac{\alpha(\sigma^2)}{p_L(\sigma^2)}} \frac{p_L(\sigma^2) \sigma^2 - \alpha(\sigma^2)}{p_L(\sigma^2) \sigma^2} \notag \\
        &= \frac{p_L(\sigma^2) \sigma^2 - \alpha(\sigma^2)}{ \left( p_L(\sigma^2)\sigma^2 -  \alpha(\sigma^2)\right) \sigma^2} = \frac{1}{\sigma^2}.
\end{align}

Thus, in total, it holds
\begin{equation}
    r_L\left(\sigma^2, \pm \, p_L(\sigma^2) \sigma^2 \right) = 
    \begin{cases}
        1 & \text{if } \, \sigma^2 \leqslant \frac{\alpha(\sigma^2)}{L p_L(\sigma^2)}, \\
        \frac{1}{ 1-L} \left( 1 -\frac{\alpha(\sigma^2)}{p_L(\sigma^2) \sigma^2} \right) & \text{if } \, \frac{\alpha(\sigma^2)}{L p_L(\sigma^2)} < \sigma^2 \leqslant \frac{\alpha(\sigma^2)}{p_L(\sigma^2)} + 1 - L, \\
        \frac{1}{\sigma^2} & \text{if } \, \sigma^2 > \frac{\alpha(\sigma^2)}{p_L(\sigma^2)} + 1 - L.
    \end{cases}
\end{equation}

\bibliographystyle{abbrv}
\bibliography{literature}

\end{document}